\documentclass[final,1p,times,numbers,sort&compress]{elsarticle}
\usepackage{txfonts}
\usepackage{latexsym}
\usepackage{enumerate}
\usepackage{amsmath,amsfonts,amscd}
\usepackage{amssymb,amsthm}
\usepackage{float}
\usepackage{mathrsfs}
\usepackage{multirow}
\usepackage{graphics}

\usepackage{subfigure}
\usepackage{epsfig}
\usepackage{booktabs}
\usepackage{epstopdf}

\numberwithin{equation}{section}

\newtheorem{theorem}{Theorem}[section]

\newtheorem{lemma}[theorem]{Lemma}
\newtheorem{definition}[theorem]{Definition}

\newtheorem{remark}[theorem]{Remark}

\journal{$\ast\ast\ast$}
\begin{document}

\begin{frontmatter}

\title{Normalized solution to the nonlinear p-Laplacian equation with an $L^2$ constrain: mass supercritical case}

\author{Yulu Tian, Deng-Shan Wang, Liang Zhao}

\begin{abstract}
\noindent
In this paper, we study the existence of ground state solutions to the following p-Laplacian equation in some dimension $N\geq3$ with an $L^2$ constraint:
\begin{equation*}
     \begin{cases}
         -\Delta_{p}u+{\vert u\vert}^{p-2}u=f(u)-\mu u \quad \text{ in } \mathbb{R}^N,\\
         {\Vert u\Vert}^2_{L^2(\mathbb{R}^N)}=m,\\
         u\in W^{1,p}(\mathbb{R}^N)\cap L^2(\mathbb{R}^N),
     \end{cases}
\end{equation*}
where $-\Delta_{p}u=div\left( {\vert\nabla u\vert}^{p-2}\nabla u \right)$, $2\leq p<N$, $f\in C(\mathbb{R},\mathbb{R})$, $m>0$, $\mu\in\mathbb{R}$ will appear as a Lagrange multiplier and the continuous nonlinearity $f$ satisfies mass supercritical conditions. We mainly study the behavior of ground state energy $E_m$ with $m>0$ changing within a certain range and aim at extending nonlinear scalar field equation when $p=2$ and reducing the constraint condition of nonlinearity $f$.
\end{abstract}


\end{frontmatter}

\section{Introduction}
In present paper, what we are concerned about is the following p-Laplacian equation with an $L^2$ constraint:
\begin{equation}\label{(P_m)}
     \begin{cases}
         -\Delta_{p}u+{\vert u\vert}^{p-2}u=f(u)-\mu u \quad \text{ in } \mathbb{R}^N,\\
         {\Vert u\Vert}^2_{L^2(\mathbb{R}^N)}=m,\\
         u\in W^{1,p}(\mathbb{R}^N)\cap L^2(\mathbb{R}^N),
     \end{cases}
\end{equation}
where $-\Delta_{p}u=div\left( {\vert\nabla u\vert}^{p-2}\nabla u \right)$ is p-Laplacian operator $2\leq p<N$, $f\in C(\mathbb{R},\mathbb{R})$, $m>0$ is a prescribed mass and $\mu\in\mathbb{R}$ will appear as a Lagrange multiplier. In particular, choosing $p=2$ and $\mu_0=\mu+1$, we obtain following nonlinear scalar field equation
\begin{equation}\label{p=2}
-\Delta u=f(u)-\mu_0 u \quad \text{ in }\mathbb{R}^N.
\end{equation}
It is worth pointing out that \eqref{p=2} closely related to a function $\psi=\psi(t,x):\mathbb{R}^+\times\mathbb{R}^N\to\mathbb{C}$, which satisfies following nonlinear Schr\"{o}dinger equation
\begin{equation}\label{NLS}
i\psi_t+\Delta\psi+g(\vert\psi\vert^2)\psi=0\text{ in }\mathbb{R}^+\times\mathbb{R}^N.
\end{equation}
This type of equation are derived as models of several physical phenomena, such as non-Newtonian fluids, dilatant fluids, electromagnetic fields and reaction diffusions. Readers may refer to \cite{Diaz1985,DiazSIAM1994} for more physical background.

Clearly,  $\psi(x,t)=e^{i\lambda t}u(x)$ is a solution to \eqref{NLS} if and only if $u(x)$ is a solution to \eqref{p=2} with $f(u)=g(\vert u\vert^2)u$. Hence, the solution of nonlinear Schr\"{o}dinger equation is transformed into the solution of the nonlinear scalar field equation \eqref{p=2}. From a physical point of view, variational method seems particularly meaningful, since the $L^2$ constraint is a reserved quantity of evolution. Moreover, the variational characteristics of such solutions are usually helpful to analyze their orbital stability \cite{BellazziniPLMS2013,CazenaveCMP1982,SoaveJDE2020,SoaveJFA2020}.

L. Jeanjean et al. first studied the mass supercritical case of \eqref{p=2} in the pioneer work \cite{JeanjeanNA1997} and proved that the relevant functional possesses the mountain pass geometric structure. Moreover, they further obtained the normalized solution to  \eqref{p=2} by a skillful compactness argument and the minimax approach. Subsequently, many researchers begin to study the normalized solutions of models closely related to \eqref{p=2}. Please turn to \cite{YeMMAS2015,YeZMP2016,Kong2022} for Kirchhoff type equations, \cite{ColinNA2004,ColinNon2010,YeZMP2021} for quasi-linear Schr\"{o}dinger equation, \cite{LiJFPTA2021,BartschJMPA2016,BartschJFA2017} for Schr\"{o}dinger systems and \cite{SoaveJDE2020,SoaveJFA2020,JeanjeanJMPA2022,JeanjeanJMPA2017} for combined nonlinearities. More results concerning normalized problems can be found in above papers and the references therein.

 Our aim is to study the existence of the ground state solutions of p-Laplacian equation \eqref{(P_m)} by using the methods inspired by \cite{JeanjeanCVPDE2020},  which constructs normalized solutions for the nonlinear scalar field equation with an $L^2$ constraint. The innovation of this paper is to generalize the Laplace operator and give the mass supercritical growth condition of the more general nonlinear term $f$ corresponding to the p-Laplacian operator. Thus, more difficulties will arise when dealing with the relationship between $p$ and $N$ (see Lemmas \ref{lem2.5} and  \ref{lem4.6} for more details).

For some suitable nonlinearities $f$, we can define the following energy functional
\begin{equation*}
I(u)=\frac{1}{p}\int_{\mathbb{R}^N}\left( {\vert\nabla u\vert}^p+{\vert u\vert}^p\right)dx-\int_{\mathbb{R}^N}F(u)dx,
\end{equation*}
which corresponds to \eqref{(P_m)} on $\mathscr{X}:= W^{1,p}(\mathbb{R}^N)\cap L^2(\mathbb{R}^N)$ with $F(t):=\int^t_0f(s)ds$ for  $t\in\mathbb{R}$. The norm of $\mathscr{X}$ is defined by
\begin{equation*}
{\Vert u\Vert}_\mathscr{X}:={\Vert u\Vert}_{W^{1,p}(\mathbb{R}^N)}+{\Vert u\Vert}_{L^2(\mathbb{R}^N)}=\left(\int_{\mathbb{R}^N}({\vert\nabla u\vert}^p+{\vert u\vert}^p)dx\right)^\frac{1}{p}+\left(\int_{\mathbb{R}^N}{\vert u\vert}^2dx\right)^{\frac{1}{2}}.
\end{equation*}

Motivated by the pioneer work \cite{JeanjeanNA1997} and a further work \cite{JeanjeanCVPDE2020} that focused on the supercritical case, we make following assumptions on nonlinearity $f$:

$(f0)\ \ f:\mathbb{R}\to\mathbb{R}$ is continuous;

$(f1)\ \ \lim\limits_{t\to0} \frac{f(t)}{{\vert t\vert}^{p-1+\frac{2p}{N}}}=0$;

$(f2)\ \ \lim\limits_{t\to\infty}\frac{f(t)}{{\vert t\vert}^{\frac{Np-N+p}{N-p}}}=0$;

$(f3)\ \ \lim\limits_{t\to\infty}\frac{F(t)}{{\vert t\vert}^{p+\frac{2p}{N}}}=+\infty$;

$(f4)\ \ t\mapsto\frac{\widetilde{F}(t)}{{\vert t\vert}^{p+\frac{2p}{N}}}$ is strictly
decreasing on $(-\infty,0)$ and strictly increasing on $(0,+\infty)$ with $\widetilde{F}(t):=f(t)t-2F(t)$.

For any given $m>0$, set $S_m:= \left\{u\in\mathscr{X}:{\Vert u\Vert}^2_{L^2(\mathbb{R}^N)}=m\right\}$. Consider the Pohozaev manifold
\begin{equation*}
     \mathcal{P}_m:=\left\{u\in S_m:P(u)=0\right\},
\end{equation*}
where $P(u)$ is the Pohozaev functional defined by
\begin{equation*}
P(u):=\frac{N(p-2)+2p}{2p}{\Vert\nabla u\Vert}^p_{L^p(\mathbb{R}^N)}+\frac{N(p-2)}{2p}{\Vert u\Vert}^p_{L^p(\mathbb{R}^N)}-\frac{N}{2}\int_{\mathbb{R}^N}\widetilde{F}(u)dx=0.
\end{equation*}
Then the ground state energy is given by $E_m:=\inf_{u\in\mathcal{P}_m}I(u)$.

To simplify the notation, set $p_*:=p+\frac{2p}{N}$, $p_c:=\frac{N(p-2)}{2}$ and $p^*:=\frac{Np}{N-p}$. To ensure that the Lagrange mutipliers are positive, we also provide the following condition:

$(f5)\ \ \text{There exists some } \gamma\in(p_*,p^*) \text{ such that } f(t)t\leq\gamma F(t)$ for all $t\in\mathbb{R}\setminus\{0\}$.

As an example for the nonlinearity which satisfies $(f0)-(f5)$, we have following odd function
\begin{equation*}
f(t):=\left[p_*\ln(1+\vert t\vert^\alpha)+\frac{\alpha\vert t\vert^\alpha}{1+\vert t\vert^\alpha}\right]\vert t\vert^{p_*-2}t,
\end{equation*}
and the primitive function is $F(t)=\vert t\vert^{p_*}\ln(1+\vert t\vert^\alpha)$, where $\alpha=\frac{1}{2}\left(\gamma-p_*\right)$.

Based on above conditions, our main results can be presented as follows.

\begin{theorem}\label{theo1.1}
Let $N\geq3$ and $f$ satisfy $(f0)-(f5)$. Then there exists $m_0>0$ small enough such that
\eqref{(P_m)} admits a ground state for any $m\in(0,m_0)$. Furthermore, if $f$ is an odd function, any ground state of \eqref{(P_m)} is non-negative for any $m\in(0,m_0)$. In both cases, the associated Lagrange multiplier $\mu$ is positive for any ground state.
\end{theorem}

\begin{remark}
We would like to give some further explanations of Theorem \ref{theo1.1}.

(i) Since $p_*<p^*$, we have $p>\frac{2N}{N+2}$, which implies that $N\geq3$. For $p=2$, the range of $N$ can be extended to $[1,\infty)$ (see \cite{JeanjeanCVPDE2020}).

(ii) For the sake of mathematical technique (see Lemma \ref{Lemma 3}), we focus on \eqref{(P_m)} with the condition ${\Vert u\Vert}^2_{L^2(\mathbb{R}^N)}=m$ instead of ${\Vert u\Vert}^p_{L^p(\mathbb{R}^N)}=m$ for any given $m>0$.
\end{remark}

The changing behavior of the ground state energy $E_m$, such as continuity and monotonicity, plays an important role in the proof of Theorem \ref{theo1.1}. Here are some results of the ground state energy $E_m$.

\begin{theorem}\label{theo1.2}
Let $N\geq3$ and $f$ satisfy $(f0)-(f4)$. Then the function $m\mapsto E_m$ is continuous, nonincreasing with $\lim_{m\to0^+}E_m=+\infty$. Moreover, if $f$ also satisfies $(f5)$, then $E_m$ is strictly decreasing with respect to $m$ in $(0,\infty)$.
\end{theorem}

The content of this paper is organized as follows.  In section 2, we prepare basic results about the nonlinearity $f$ and the energy functional $I$ that will be needed later. In section 3, we mainly focus on the properties of ground state energy $E_m$. Section 4 is devoted to the proofs the Theorems \ref{theo1.1} and \ref{theo1.2}.

\section{Basic and preliminary results}
In this section, we give some basic results of the nonlinearity $f$ and the functional $I$ that are useful in the proof of our main results. Set
\begin{equation*}
B_m:=\left\{u\in\mathscr{X}:{\Vert u\Vert}^2_{L^2(\mathbb{R}^N)}\leq m\right\},
\end{equation*}
where $m>0$ is arbitrary but fixed. Then we have following results.

\begin{lemma}\label{lem2.1}
 Assume $N\geq3$ and $f$ satisfies $(f0)-(f2)$, then the following statements hold.

(i) There exists $\delta=\delta(N,p,m)>0$ small enough such that for all $u\in B_m$ satisfying ${\Vert\nabla u\Vert}_{L^p(\mathbb{R}^N)}\leq\delta$, we have
\begin{equation*}
\frac{1}{2p}\int_{\mathbb{R}^N}({\vert\nabla u\vert}^p+{\vert u\vert}^p)dx\leq I(u)\leq\int_{\mathbb{R}^N}({\vert\nabla u\vert}^p+{\vert u\vert}^p)dx.
\end{equation*}

(ii) Let $\{u_n\}$ be a bounded sequence in $\mathscr{X}$ with $\lim_{n\to\infty}{\Vert u_n\Vert}_{L^{p_*}(\mathbb{R}^N)}=0$, then
\begin{equation*}
\lim\limits_{n\to\infty}\int_{\mathbb{R}^N}F(u_n)dx=0 =\lim\limits_{n\to\infty}\int_{\mathbb{R}^N}\widetilde{F}(u_n)dx.
\end{equation*}

(iii) Let $\{u_n\}$ and $\{v_n\}$ be two bounded sequences in $\mathscr{X}$ with $\lim_{n\to\infty}{\Vert v_n\Vert}_{L^{p_*}(\mathbb{R}^N)}=0$, then
\begin{equation}\label{eq2.1}
\lim\limits_{n\to\infty}\int_{\mathbb{R}^N}F(u_n)v_ndx=0.
\end{equation}
\end{lemma}

\begin{proof}
(i) After transposition and calculation of terms, it only remains to prove that there exists $\delta=\delta(N,p,m)>0$ small enough such that, for any $u\in B_m$ with ${\Vert\nabla u\Vert}_{L^p(\mathbb{R}^N)}\leq\delta$,
\begin{equation}\label{eq2.2}
\int_{\mathbb{R}^N}\vert F(u)\vert dx\leq\frac{1}{2p}\int_{\mathbb{R}^N}{\vert\nabla u\vert}^pdx.
\end{equation}
For any arbitrary constant $\varepsilon>0$, it follows from $(f0)-(f2)$ that there exists $C_{\varepsilon}>0$ such that $\vert F(t)\vert\leq\varepsilon{\vert t\vert}^{p_*}+C_{\varepsilon}{\vert t\vert}^{p^*}$ for $t\in\mathbb{R}$. For any $u\in B_m$, using H{\"o}lder inequality and Gagliardo-Nirenberg inequality, we have
\begin{equation*}
     \begin{aligned}
          \int_{\mathbb{R}^N}\vert F(u)\vert dx & \leq\varepsilon\int_{\mathbb{R}^N}{\vert u\vert}^{p_*}dx+C_{\varepsilon}\int_{\mathbb{R}^N}{\vert u\vert}^{p^*}dx\\
          & \leq\varepsilon m^{\frac{p}{N}}\left(\int_{\mathbb{R}^N}{\vert u\vert}^{p^*}dx\right)^{\frac{N-p}{N}}+C_{\varepsilon}\int_{\mathbb{R}^N}{\vert u\vert}^{p^*}dx\\
          & \leq\varepsilon m^{\frac{p}{N}}C_1\int_{\mathbb{R}^N}{\vert\nabla u\vert}^pdx+C_{\varepsilon}C_2\left(\int_{\mathbb{R}^N}{\vert\nabla u\vert}^pdx\right)^{\frac{N}{N-p}}\\
          & =\left[\varepsilon m^{\frac{p}{N}}C_1+C_{\varepsilon}C_2\left(\int_{\mathbb{R}^N}{\vert\nabla u\vert}^pdx\right)^{\frac{p}{N-p}} \right] \int_{\mathbb{R}^N}{\vert\nabla u\vert}^pdx,
     \end{aligned}
\end{equation*}
where $C_1$ and $C_2$ are positive constants depending only on $p$ and $N$. Clearly, we can choose $\varepsilon\leq\frac{1}{4pC_1m^{\frac{p}{N}}}$ and $\delta\leq\left(\frac{1}{4pC_{\varepsilon}C_2}\right)^{\frac{N-p}{p^2}}$ to obtain \eqref{eq2.1}.

(ii) Now we only prove the second part of the claim, that is,
\begin{equation}\label{eq2.3}
     \lim\limits_{n\to\infty}\int_{\mathbb{R}^N}\widetilde{F}(u_n)dx=0\text{ if } \lim\limits_{n\to\infty}{\Vert u_n\Vert}_{L^{p_*}(\mathbb{R}^N)}=0.
\end{equation}
The rest of the claim can be proved in a similar way. Choose $M>0$ large enough such that $\sup_{n\geq1}{\Vert u_n\Vert}_\mathscr{X}\leq M$, then ${\Vert\nabla u_n\Vert}^p_{L^p(\mathbb{R}^N)}\leq M^p$. For any arbitrary constant $\varepsilon>0$, it follows from $(f0)-(f2)$ that there exists $M_{\varepsilon}>0$ such that
\begin{equation*}
     \vert\widetilde{F}(t)\vert\leq\vert f(t)\vert\vert t\vert+\vert F(t)\vert\leq\varepsilon{\vert t\vert}^{p^*}+M_{\varepsilon}{\vert t\vert}^{p_*}\text{ for }t\in\mathbb{R}.
\end{equation*}
Hence, direct calculations yield that
\begin{equation*}
     \begin{aligned}
          \int_{\mathbb{R}^N}\vert\widetilde{F}(u)\vert dx & \leq\varepsilon\int_{\mathbb{R}^N}{\vert u_n\vert}^{p^*}dx+M_{\varepsilon}\int_{\mathbb{R}^N}{\vert u_n\vert}^{p_*}dx\\
          & \leq\varepsilon C_1\left(\int_{\mathbb{R}^N}{\vert\nabla u_n\vert}^pdx\right)^{\frac{N}{N-p}}+M_{\varepsilon}\int_{\mathbb{R}^N}{\vert u_n\vert}^{p_*}dx\\
          & \leq\varepsilon C_1M^{p^*}+M_{\varepsilon}\int_{\mathbb{R}^N}{\vert u_n\vert}^{p_*}dx.
     \end{aligned}
\end{equation*}
By the arbitrariness of the $\varepsilon$, \eqref{eq2.2} holds.

(iii) Choose $M_1$, $M_2>0$ large enough such that $\sup_{n\geq1}{\Vert u_n\Vert}_\mathscr{X}\leq M_1$ and $\sup_{n\geq1}{\Vert v_n\Vert}_\mathscr{X}\leq M_2$, then ${\Vert u_n\Vert}^2_{L^2(\mathbb{R}^N)}\leq M_1^2$, ${\Vert\nabla u_n\Vert}^p_{L^p(\mathbb{R}^N)}\leq M_1^p$ and ${\Vert\nabla v_n\Vert}^p_{L^p(\mathbb{R}^N)}\leq M_2^p$. For any arbitrary constant $\varepsilon>0$, it follows from $(f0)-(f2)$ that there exists $M_{\varepsilon}>0$ such that $\vert F(t)\vert\leq\varepsilon{\vert t\vert}^{p^*-1}+M_{\varepsilon}{\vert t\vert}^{p_*-1}$ for $t\in\mathbb{R}$.
Using H{\"o}lder inequality and Gagliardo-Nirenberg inequality, we have
\begin{equation*}
     \begin{aligned}
          \int_{\mathbb{R}^N}\vert f(u_n)v_n\vert dx & \leq\varepsilon\int_{\mathbb{R}^N}{\vert u_n\vert}^{p^*-1}\vert v_n\vert dx+M_{\varepsilon}\int_{\mathbb{R}^N}{\vert u_n\vert}^{p_*-1}\vert v_n\vert dx\\
          & \leq\varepsilon\left(\int_{\mathbb{R}^N}{\vert u_n\vert}^{p^*}dx\right)^{p^*-1}\left(\int_{\mathbb{R}^N} {\vert v_n\vert}^{p^*}dx\right)^{p^*}\\
          & +M_{\varepsilon}\left(\int_{\mathbb{R}^N}{\vert u_n\vert}^{p_*}dx\right)^{\frac{Np-N+2p}{Np+2p}} \left(\int_{\mathbb{R}^N}{\vert v_n\vert}^{p_*}dx\right)^{\frac{N}{(N+2)p}}\\
          & \leq\varepsilon C_1C_2\left(\int_{\mathbb{R}^N}{\vert\nabla u_n\vert}^pdx\right)^{\frac{N(Np-N+2p)}{p(N+2)(N-p)}}\left(\int_{\mathbb{R}^N} {\vert\nabla v_n\vert}^pdx\right)^{\frac{1}{p}}\\
          & +M_{\varepsilon}\left(\int_{\mathbb{R}^N}{\vert u_n\vert}^{p_*}dx\right)^{\frac{Np-N+2p}{Np+2p}}{\Vert v_n\Vert}_{L^{p_*}(\mathbb{R}^N)}\\
          & \leq\varepsilon C_1C_2M_1^{\frac{N(Np-N+2p)}{(N+2)(N-p)}}M_2+M_{\varepsilon}M_1^{\frac{2(Np-N+2p)}{N(N+2)}}\left(\int_{\mathbb{R}^N}{\vert\nabla u_n\vert}^pdx\right)^{\frac{Np-N+2p}{Np+2p}}{\Vert v_n\Vert}_{L^{p_*}(\mathbb{R}^N)}\\
          & \leq\varepsilon C_1C_2M_1^{\frac{N(Np-N+2p)}{(N+2)(N-p)}}M_2+M_{\varepsilon} M_1^{\frac{Np-N+2p}{N}}{\Vert v_n\Vert}_{L^{p_*}(\mathbb{R}^N)}.
     \end{aligned}
\end{equation*}
Hence, \eqref{eq2.3} holds.
\end{proof}

\begin{remark}\label{re2.1}
Under the assumptions of Lemma \ref{lem2.1} and with slight modification of the proof of \eqref{eq2.1}, there exists $\delta=\delta(N,p,m)>0$ small enough such that for each $u\in B_m$ with ${\Vert\nabla u_n\Vert}_{L^p(\mathbb{R}^N)}\leq \delta$, it holds
\begin{equation*}
     \int_{\mathbb{R}^N}\vert\widetilde{F}(u)\vert dx\leq\frac{2}{Np}\int_{\mathbb{R}^N}{\vert\nabla u_n\vert}^pdx,
\end{equation*}
which further leads to
\begin{equation*}
     \begin{aligned}
          P(u) & =\frac{N(p-2)+2p}{2p}\int_{\mathbb{R}^N}{\vert\nabla u\vert}^pdx+\frac{N(p-2)}{2p}\int_{\mathbb{R}^N}{\vert u\vert}^pdx-\frac{N}{2}\int_{\mathbb{R}^N}\widetilde{F}(u)dx\\
          & \geq\frac{N(p-2)+2(p-1)}{2p}\int_{\mathbb{R}^N}{\vert\nabla u\vert}^pdx+\frac{N(p-2)}{2p}\int_{\mathbb{R}^N}{\vert u\vert}^pdx\\
          & \geq\frac{N(p-2)+2(p-1)}{2p}\int_{\mathbb{R}^N}{\vert\nabla u\vert}^pdx.
     \end{aligned}
\end{equation*}
\end{remark}

\begin{remark}\label{re2.2}
Assume $N\geq3$ and $f$ satisfies $(f0)$, $(f1)$ and $(f4)$, then we can define
\begin{equation}\label{eq2.7}
     \Lambda(t):=\left\{
     \begin{aligned}
          \frac{f(t)t-2F(t)}{{\vert t\vert}^{p_*}},
          &\quad \text{ for } t\neq0,\\
          0,
          &\quad \text{ for } t=0.
     \end{aligned}
     \right.
\end{equation}
Clearly, the function $\Lambda:\mathbb{R}\to\mathbb{R}$ is continuous and strictly decreasing on $(-\infty,0)$ and strictly increasing on $(0,+\infty)$.
\end{remark}

\begin{lemma}\label{lem2.3}
Assume $N\geq3$ and $f$ satisfies $(f0)-(f4)$, then $f(t)t>p_*F(t)>0$ for all $t\neq0$.
\end{lemma}

\begin{proof}
For clarity, we divide the proof into five claims.

\textbf{Claim 1.} $F(t)>0$ for all $t\neq0$.

Suppose that there exists $t_0\neq0$ such that $F(t_0)\leq0$. From $(f1)$ and $(f3)$, the function $F(t)/{\vert t\vert}^{p_*}$ reaches its global minimum at some $\theta\neq0$ satisfying $F(\theta)\leq0$ and
\begin{equation*}
     \left[\frac{F(t)}{{\vert t\vert}^{p_*}}\right]^{'}_{t=\theta}=\frac{f(\theta)\theta-p_*F(\theta)}{{\vert \theta\vert}^{p_*+1}\text{sign}(\theta)}=0.
\end{equation*}
Noting that by Remark \ref{re2.2}, we have $f(t)t>2F(t)$ for any $t\neq0$, then
\begin{equation*}
     0<f(\theta)\theta-2F(\theta)=\left(p-2+2p/N\right)F(\theta)\leq0.
\end{equation*}
Therefore, Claim 1 is proved.

\textbf{Claim 2.} There exist a positive sequence $\{\theta^+_n\}$ and a negative sequence $\{\theta^-_n\}$ such that $\vert \theta^\pm_n\vert\to0$ as $n\to+\infty$ and $f(\theta^\pm_n)\theta^\pm_n>p_*F(\theta^\pm_n)$ for each $n\geq1$.

We only consider the positive case by contradiction, the negative case can be proved similarly. Assume that there exists $T_\theta>0$ small enough such that $f(t)t\leq p_*F(t)$ for any $t\in(0,T_\theta]$. By Claim 1 we obtain that
\begin{equation*}
     \frac{F(t)}{t^{p_*}}\geq\frac{F(T_\theta)}{{T_\theta}^{p_*}}>0 \text{ for all } t\in(0,T_\theta],
\end{equation*}
which contradicts with the fact that $\lim_{t\to0}F(t)/{\vert t\vert}^{p_*}=0$ from $(f1)$. Hence, the desired sequence $\{\theta^+_n\}$ exists.

\textbf{Claim 3.} There exist a positive sequence $\{\lambda^+_n\}$ and a negative sequence $\{\lambda^-_n\}$ such that $\vert \lambda^\pm_n\vert\to+\infty$ as $n\to+\infty$ and $f(\lambda^\pm_n)\lambda^\pm_n>p_*F(\lambda^\pm_n)$ for each $n\geq1$.

We only prove the existence of $\{\lambda^-_n\}$, the positive case can proved similarly. Assume that there exists $T_\lambda>0$ such that $f(t)t\leq p_*F(t)$ for any $t\leq-T_\lambda$, then
\begin{equation*}
     \frac{F(t)}{t^{p_*}}\leq\frac{F(-T_\lambda)}{{T_\lambda}^{p_*}}<+\infty\text{ for all } t<-T_\lambda,
\end{equation*}
which contradicts with the condition $(f3)$. Hence the proof of Claim 3 is completed.

\textbf{Claim 4.} $f(t)t\geq p_*F(t)$ for all $t\neq0$.

Suppose $f(t_\tau)t_\tau<(p_*)F(t_\tau)$ for some $t_\tau\neq0$. Since the case $t_\tau>0$ can be treated in a similar way, we further suppose that $t_\tau<0$. By Claims 2 and 3, there exist $\theta_{min}$ and $\theta_{max}\in\mathbb{R}$ such that $\theta_{min}<t_0\tau<\theta_{max}<0$ and
\begin{equation}\label{eq2.8}
     f(t)t<p_*F(t)\text{ for } t\in(\theta_{min},\theta_{max}).
\end{equation}
Moreover,
\begin{equation}\label{eq2.9}
     f(\theta_{min})\theta_{min}=p_*F(\theta_{min})\text{ and }f(\theta_{max})\theta_{max}=p_*F(\theta_{max}).
\end{equation}
It follows from \eqref{eq2.8} and Claim 1 that
\begin{equation}\label{eq2.10}
     \frac{F(\theta_{min})}{{\vert\theta_{min}\vert}^{p_*}}\leq\frac{F(\theta_{max})}{{\vert\theta_{max}\vert}^{p_*}}.
\end{equation}
On the other hand, using \eqref{eq2.9} and $(f4)$, we have
\begin{equation*}\label{eq2.11}
     \frac{F(\theta_{min})}{{\vert\theta_{min}\vert}^{p_*}}=\frac{N}{N(p-2)+2p}\frac{\widetilde{F}(\theta_{min})}{{\vert\theta_{min}\vert}^{p_*}}>\frac{N}{N(p-2)+2p}\frac{\widetilde{F}(\theta_{max})}{{\vert\theta_{max}\vert}^{p_*}}=\frac{F(\theta_{max})}{{\vert\theta_{max}\vert}^{p_*}}.
\end{equation*}
This contradicts with \eqref{eq2.10}, which implies Claim 4.

\textbf{Claim 5.} $f(t)t>p_*F(t)$ for any $t\neq0$.

From Claim 4, the function $F(t)/t^{p_*}$ is nonincreasing with respect to $t$ on $(-\infty,0)$ and nondecreasing on $(0,+\infty)$. Hence, it follows from $(f4)$ that the function $f(t)/t^{p_*-1}$ is strictly increasing with respect to $t$ on $(-\infty,0)\cup(0,+\infty)$. Then
\begin{equation*}
     p_*F(t)=p_*\int^t_0f(s)ds<p_*\frac{f(t)}{{\vert t\vert}^{p_*-1}}\int^t_0\vert s\vert^{p_*-1}ds=f(t)t\text{ for }t\neq0,
\end{equation*}
which leads to Claim 5.

To sum up, using Claims 1 and 5, we complete the proof of Lemma \ref{lem2.3}.
\end{proof}

Recalling from the fiber map, for any $u\in\mathscr{X}$ and $s\in\mathbb{R}$, we define
\begin{equation*}
     u(x)\mapsto(s\star u)(x):=e^{\frac{Ns}{2}}u(e^sx) \text{ for a.e. } x\in\mathbb{R}^N.
\end{equation*}
Clearly, $s\star u\in\mathscr{X}$ and it preserves the $L^2$-norm ${\Vert s\star u\Vert}_{L^2(\mathbb{R}^N)}={\Vert u\Vert}_{L^2(\mathbb{R}^N)}$ for $s\in\mathbb{R}$.
Moreover, we have
\begin{equation}\label{bian_huan}
     \begin{aligned}
          \int_{\mathbb{R}^N}{\vert\nabla (s\star u)\vert}^pdx
          & =(e^s)^{p_c+p}\int_{\mathbb{R}^N}{\vert\nabla u\vert}^pdx,\\
          \int_{\mathbb{R}^N}{\vert s\star u\vert}^pdx
          & =(e^s)^{p_c}\int_{\mathbb{R}^N}{\vert u\vert}^pdx.
     \end{aligned}
\end{equation}
The following is an important result of the function $s\mapsto I(s\star u)$.

\begin{lemma}\label{lem2.2}
Assume $N\geq3$ and $f$ satisfies the conditions $(f0)-(f3)$. Then for all $u\in\mathscr{X}\setminus\{0\}$, we have

(i) $I(s\star u)\mapsto 0^+$ as $s\mapsto-\infty$;

(ii) $I(s\star u)\mapsto -\infty$ as $s\mapsto+\infty$.
\end{lemma}
\begin{proof}
(i) Set $m_0:={\Vert u\Vert}^2_{L^2(\mathbb{R}^N)}>0$, then $s\star u\in S_{m_0}\subset B_{m_0}$. Thanks to Lemma \ref{lem2.1}(i), when $s\to-\infty$, it holds
\begin{equation*}
     \frac{1}{2p}(e^s)^{p_c+p}\int_{\mathbb{R}^N}{\vert\nabla u\vert}^pdx+\frac{1}{2p}e^{sp_c}\int_{\mathbb{R}^N}{\vert u\vert}^pdx\leq I(s\star u)\leq(e^s)^{p_c+p}\int_{\mathbb{R}^N}{\vert\nabla u\vert}^pdx+e^{sp_c}\int_{\mathbb{R}^N}{\vert u\vert}^pdx,
\end{equation*}
which implies $\lim_{s\to-\infty}I(s\star u)=0^+$.

(ii) For each $c\geq0$, we define the following auxiliary function:
\begin{equation}\label{eq2.5}
     \Gamma_c(t):=\left\{
     \begin{aligned}
          \frac{F(t)}{{\vert t\vert}^{p_*}}+c,
          &  \quad\text{ for } t\neq0,\\
          c,&  \quad\text{ for } t=0.
     \end{aligned}
     \right.
\end{equation}
Clearly, $F(t)=\Gamma_c(t){\vert t\vert}^{p_*}-c{\vert t\vert}^{p_*}$ for all $t\in\mathbb{R}$. It follows from $(f0)-(f3)$ that $\Gamma_c(t)$ is continuous with respect to $t$ and $\lim_{\vert t\vert\to\infty}\Gamma_c(t)=+\infty$. Thus there exists $c>0$ large enough such that $\Gamma_c(t)\geq0$ for any $t\in\mathbb{R}$. Thanks to the Fatou's lemma, we have
\begin{equation*}
     \lim\limits_{s\to+\infty}\int_{\mathbb{R}^N}\Gamma_c(e^{\frac{Ns}{2}}u){\vert u\vert}^{p_*}dx=+\infty.
\end{equation*}
Since
\begin{equation}\label{eq2.6}
     \begin{aligned}
          I(s\star u) &=\frac{1}{p}\int_{\mathbb{R}^N}({\vert\nabla(s\star u)\vert}^p+{\vert s\star u\vert}^p)dx+c\int_{\mathbb{R}^N}{\vert s\star u\vert}^{p_*}dx-\int_{\mathbb{R}^N}\Gamma_c(s\star u){\vert s\star u\vert}^{p_*}dx\\
          & =\frac{1}{p}(e^s)^{p_c+p}\int_{\mathbb{R}^N}{\vert\nabla u\vert}^pdx+\frac{1}{p}e^{sp_c}\int_{\mathbb{R}^N}{\vert u\vert}^pdx+(e^s)^{p_c+p}\left(c\int_{\mathbb{R}^N}{\vert
          u\vert}^{p_*}-\Gamma_c(e^{\frac{Ns}{2}}u){\vert u\vert}^{p_*}dx\right)\\
          & =(e^s)^{p_c+p}\left[\frac{1}{p}\int_{\mathbb{R}^N}{\vert\nabla u\vert}^pdx+\frac{1}{p}e^{-sp}\int_{\mathbb{R}^N}{\vert u\vert}^pdx+c\int_{\mathbb{R}^N}{\vert u\vert}^{p_*}dx-\int_{\mathbb{R}^N}\Gamma_c(e^{\frac{Ns}{2}}u){\vert u\vert}^{p_*}dx\right],
     \end{aligned}
\end{equation}
then $I(s\star u)\to-\infty$ as $s\to+\infty$.
\end{proof}

Recalling from the Pohozaev functional
\begin{equation*}
     P(u)=\frac{p_c+p}{p}\int_{\mathbb{R}^N}{\vert\nabla u\vert}^pdx+\frac{p_c}{p}\int_{\mathbb{R}^N}{\vert u\vert}^pdx-\frac{N}{2}\int_{\mathbb{R}^N}\widetilde{F}(u)dx,
\end{equation*}
where $\widetilde{F}(t):=f(t)t-2F(t)$ for any $t\in\mathbb{R}$, and thanks to the monotonicity condition $(f4)$, we have following lemma.
\begin{lemma}\label{lem2.4}
Assume $N\geq3$ and $f$ satisfies $(f0)-(f4)$. Then for any $u\in\mathscr{X}\backslash\{0\}$, following statements are true:

(i) There exists a unique number $s_0=s(u)\in\mathbb{R}$ such that $P(s_0\star u)=0$;

(ii) $I(s_0\star u)>I(s\star u)$ for any $s\neq s_0$. Moreover, $I(s_0\star u)>0$;

(iii) The mapping $u\mapsto s_0=s(u)$ is continuous in $u\in\mathscr{X}\backslash\{0\}$;

(iv) $s(u(\cdot+y))=s(u)$ for any $y\in\mathbb{R}^N$. Moreover if $f$ is odd, then we have $s(-u)=s(u)$.
\end{lemma}
\begin{proof}
(i) Since
\begin{equation*}
     \begin{aligned}
          I(s\star u) &=\frac{1}{p}\int_{\mathbb{R}^N}{\vert\nabla(s\star u)\vert}^pdx+\frac{1}{p}\int_{\mathbb{R}^N}{\vert s\star u\vert}^pdx-\int_{\mathbb{R}^N}F(s\star u)dx\\
          & =\frac{1}{p}(e^s)^{p_c+p}\int_{\mathbb{R}^N}{\vert\nabla u\vert}^pdx+\frac{1}{p}e^{sp_c}\int_{\mathbb{R}^N}{\vert u\vert}^pdx-e^{-sN}\int_{\mathbb{R}^N}F(e^{\frac{Ns}{2}}u)dx,
     \end{aligned}
\end{equation*}
\begin{equation*}
     \begin{aligned}
          P(s\star u) &=\frac{p_c+p}{p}\int_{\mathbb{R}^N}{\vert\nabla(s\star u)\vert}^pdx+\frac{p_c}{p}\int_{\mathbb{R}^N}{\vert s\star u\vert}^pdx-\frac{N}{2}\int_{\mathbb{R}^N}\widetilde{F}(s\star u)dx\\
          & =\frac{p_c+p}{p}(e^s)^{p_c+p}\int_{\mathbb{R}^N}{\vert\nabla u\vert}^pdx+\frac{p_c}{p}e^{sp_c}\int_{\mathbb{R}^N}{\vert u\vert}^pdx-\frac{N}{2}e^{-sN}\int_{\mathbb{R}^N}\widetilde{F}(e^{\frac{Ns}{2}}u)dx,
     \end{aligned}
\end{equation*}
\begin{equation*}
     \begin{aligned}
          \frac{d}{ds}I(s\star u) &=\frac{p_c+p}{p}(e^s)^{p_c+p}\int_{\mathbb{R}^N}{\vert\nabla u\vert}^pdx+\frac{p_c}{p}e^{sp_c}\int_{\mathbb{R}^N}{\vert u\vert}^pdx-\frac{N}{2}e^{-sN}\int_{\mathbb{R}^N}\widetilde{F}(e^{\frac{Ns}{2}}u)dx\\
          & =P(s\star u).
     \end{aligned}
\end{equation*}
By Lemma \eqref{lem2.2}, we have
\begin{equation*}
     \lim\limits_{s\to-\infty}I(s\star u)=0^+, \qquad \lim\limits_{s\to+\infty}I(s\star u)=-\infty.
\end{equation*}
Hence, $I(s\star u)$ reaches its global maximum at some $s_0=s(u)\in\mathbb{R}$, then
\begin{equation*}
     P(s_0\star u)=\frac{d}{ds}I(s_0\star u)=0.
\end{equation*}
Note that $\widetilde{F}(t)=\Lambda(t){\vert t\vert}^{p_*}$ for any $t\in\mathbb{R}$,
where $\Lambda$ is defined by \eqref{eq2.7}, we have
\begin{equation*}
     \begin{aligned}
          P(s\star u) &=\frac{p_c+p}{p}(e^s)^{p_c+p}\int_{\mathbb{R}^N}{\vert\nabla u\vert}^pdx+\frac{p_c}{p}e^{sp_c}\int_{\mathbb{R}^N}{\vert u\vert}^pdx-\frac{N}{2}e^{-sN}\int_{\mathbb{R}^N}\widetilde{F}(e^{\frac{Ns}{2}}u)dx\\
          & =\frac{p_c+p}{p}(e^s)^{p_c+p}\int_{\mathbb{R}^N}{\vert\nabla u\vert}^pdx+\frac{p_c}{p}e^{sp_c}\int_{\mathbb{R}^N}{\vert u\vert}^pdx-\frac{N}{2}e^{-sN}\int_{\mathbb{R}^N}\Lambda(e^{\frac{Ns}{2}}u)e^{\frac{(N+2)sp}{2}}{\vert u\vert}^{p_*}dx\\
          & =(e^s)^{p_c+p}\left[\frac{p_c+p}{p}\int_{\mathbb{R}^N}{\vert\nabla u\vert}^pdx+\frac{p_c}{p}e^{-sp}\int_{\mathbb{R}^N}{\vert u\vert}^pdx-\frac{N}{2}\int_{\mathbb{R}^N}\Lambda(e^{\frac{Ns}{2}}u){\vert u\vert}^{p_*}dx\right]
     \end{aligned}
\end{equation*}
Thanks to $(f4)$ and Remark \ref{re2.2}, for any fixed $t\in\mathbb{R}\backslash\{0\}$, the function $s\mapsto\Lambda(e^{Ns/2}t)$ is strictly increasing. Hence, $s_0=s(u)$ is unique and the mapping $u\mapsto s(u)$ is well-defined.

(ii) The result is the direct conclusion proved above.

(iii) Let $u\in\mathscr{X}\backslash\{0\}$ and $\{u_n\}\subset\mathscr{X}\backslash\{0\}$ be any sequence such that $u_n\to u$ as $n\to\infty$ in $\mathscr{X}$.
Set $s_n:=s(u_n)$ for any $n\geq1$, then we only need to prove that, up to a subsequence if necessary, $s_n\to s_0=s(u)$ as $n\to+\infty$.

Firstly, we prove that $\{s_n\}$ is bounded. Clearly, $\Gamma_c$ defined in \eqref{eq2.5} is a continuous coercive function. Moreover, by Lemma \eqref{lem2.3}, $\Gamma_0(t)\geq0$ for any $t\in\mathbb{R}$. If $s_n\to+\infty$, up to a subsequence, by Fatou's lemma and the fact that $u_n\to u$ a.e. in $\mathbb{R}^N$ and $u\neq0$, then we have
\begin{equation*}
     \lim\limits_{n\to+\infty}\int_{\mathbb{R}^N}\Gamma_0(e^{\frac{Ns_n}{2}}u_n){\vert u_n\vert}^{p_*}dx=+\infty.
\end{equation*}
Using the conclusion (ii) and \eqref{eq2.6} with $c=0$, we know that as $n\to\infty$,
\begin{equation}\label{eq2.12}
          0\leq(e^{-s_n})^{p_c+p}I(s_n\star u_n)=\frac{1}{p}\int_{\mathbb{R}^N}{\vert\nabla u\vert}^pdx+\frac{1}{p}e^{-ps_n}\int_{\mathbb{R}^N}{\vert u_n\vert}^pdx-\int_{\mathbb{R}^N}\Gamma_0(e^{\frac{Ns_n}{2}}u_n){\vert u_n\vert}^{p_*}dx\to-\infty,
\end{equation}
which is impossible. Then $\{s_n\}$ is bounded from above. On the other hand, (ii) implies that
\begin{equation*}
     I(s_n\star u_n)\geq I(s_0\star u_n) \quad \text{ for any } n\geq1.
\end{equation*}
Since $s_0\star u_n\to s_0\star u$ in $\mathscr{X}$, it follows that $I(s_0\star u_n)=I(s_0\star u)+o_n(1)$
and
\begin{equation}\label{eq2.13}
     \liminf\limits_{n\to+\infty}I(s_n\star u_n)\geq I(s_0\star u)>0.
\end{equation}
We can choose $m>0$ large enough such that $\{s_n\star u_n\}\subset B_m$. From Lemma \ref{lem2.1}(i), \eqref{bian_huan} and \eqref{eq2.13}, $\{s_n\}$ is bounded from below.

Without lose of generality, we can assume that:
\begin{equation*}
     \lim_{n\to\infty}s_n=\hat{s}\text{ for some }\hat{s}\in\mathbb{R}.
\end{equation*}
Since $\lim_{n\to\infty}u_n=u$ in $\mathscr{X}$, then $s_n\star u_n\to \hat{s}\star u$ in $\mathscr{X}$. Thanks to $P(s_n\star u_n)=0$ for $n\geq1$, we have $P(\hat{s}\star u)=0$. It follows from conclusion (i) that $\hat{s}=s_0$, which implies conclusion (iii).

(iv) By changing variables in the integrals, we have
\begin{equation*}
     P(s(u)\star u(\cdot+y))=P(s(u)\star u(\cdot))=0\text{ for any }y\in\mathbb{R}^N.
\end{equation*}
It follows from conclusion (i) that $s(u(\cdot+y))=s(u)$. Moreover, if $f$ is an odd function, it holds
\begin{equation*}
     P(s(u)\star (-u))=P(-(s(u)\star u))=P(s(u)\star u)=0,
\end{equation*}
which leads to $s(-u)=s(u)$.
\end{proof}

At the end of this section, we have the following result about the Pohozaev manifold
$$\mathcal{P}_m:=\left\{u\in S_m : P(u)=0\right\}$$
and the energy functional $I$ constrained to $\mathcal{P}_m$.

\begin{lemma}\label{lem2.5}
Assume $N\geq3$ and $f$ satisfies $(f0)-(f4)$, then following statements hold:

(i) $\mathcal{P}_m\neq\emptyset$;

(ii) $\inf_{u\in\mathcal{P}_m}{\Vert\nabla u_n\Vert}_{L^p(\mathbb{R}^N)}>0$;

(iii) $\inf_{u\in\mathcal{P}_m}I(u)>0$;

(iv) $I$ is coercive on $\mathcal{P}_m$, that is, for every sequence $\{u_n\}\subset\mathcal{P}_m$,
\begin{equation*}
     {\Vert u_n\Vert}_\mathscr{X}\to+\infty\text{ implies }I(u_n)\to+\infty\text{ as }n\to\infty.
\end{equation*}
\end{lemma}

Following Lions lemma (see \cite{LionsAIHPAN1984}) is needed in the proof of Lemma \ref{lem2.5}.

\begin{lemma}\label{Lemma I.1}
Let $1<p\leq+\infty$, $1\leq q<+\infty$ with $q\neq p^*$ if $p<N$. Assume that $u_n$ is bounded in $L^{q}(\mathbb{R}^N)$ and $\nabla u_n$ is bounded in $L^p(\mathbb{R}^N)$ such  that
\begin{equation*}
    \lim_{n\to\infty}\sup\limits_{y\in\mathbb{R}^N}\int_{y+B_R}{\vert u_n\vert}^qdx=0 \text{ for some } R>0,
\end{equation*}
then $\lim_{n\to\infty}u_n=0$ in $L^\alpha(\mathbb{R}^N)$ for some $\alpha$ between $q$ and $p^*$.
\end{lemma}

\noindent\textit{The proof of Theorem \ref{lem2.5}}.
(i) This conclusion is a direct result of Lemma \ref{lem2.4}(i).

(ii) If there exists $\{u_n\}\subset\mathcal{P}_m$ such that ${\Vert\nabla u_n\Vert}_{L^p(\mathbb{R}^N)}\to0$ as $n\to\infty$, then it follows from Remark \ref{re2.1} that
\begin{equation*}
    0=P(u_n)\geq\frac{N(p-2)+2(p-1)}{2p}\int_{\mathbb{R}^N}{\vert\nabla u_n\vert}^pdx+\frac{p_c}{p}\int_{\mathbb{R}^N}{\vert u_n\vert}^pdx>0
\end{equation*}
for $n>0$ large enough, which is a contradiction. Thus, $\inf_{u\in\mathcal{P}_m}{\Vert\nabla u_n\Vert}_{L^p(\mathbb{R}^N)}>0$.

(iii) Thanks to Lemma \ref{lem2.4}, $I(u)=I(0\star u)\geq I(s\star u)$ for any $u\in\mathcal{P}_m$ and $s\in\mathbb{R}$. For any small $\delta>0$ given by Lemma \ref{lem2.1}(i), we set
\begin{equation*}
    s:=\frac{p}{p_c+p}\ln{\left(\frac{\delta}{{\Vert\nabla u_n\Vert}_{L^p(\mathbb{R}^N)}}\right)}.
\end{equation*}
A simple calculation yields that ${\Vert\nabla(s\star u_n)\Vert}_{L^p(\mathbb{R}^N)}=\delta$. Due to Lemma \ref{lem2.1}(i), we have
\begin{equation*}
    I(u)\geq I(s\star u)\geq\frac{1}{2p}\int_{\mathbb{R}^N}({\vert\nabla(s\star u)\vert}^p+{\vert s\star u\vert}^p)dx\geq\frac{1}{2p}\int_{\mathbb{R}^N}{\vert\nabla(s\star u)\vert}^pdx=\frac{1}{2p}\delta^p>0.
\end{equation*}
Then conclusion (iii) holds.

(iv) Suppose there exists a sequence $\{u_n\}\subset\mathcal{P}_m$ and a constant $C>0$ such that
$$\lim\limits_{n\to\infty}{\Vert u_n\Vert}_\mathscr{X}=+\infty\text{ and }\sup\limits_{n\geq1}I(u_n)\leq C.$$
 We consider the following three cases.

\textbf{Case 1}: For any $n\geq1$, there exists a large number $M_1>0$ such that $\int_{\mathbb{R}^N}{\vert u_n\vert}^pdx\leq M_1$. Note that $\lim_{n\to\infty}{\Vert u_n\Vert}_\mathscr{X}=\infty$, then $\lim_{n\to\infty}\int_{\mathbb{R}^N}{\vert\nabla u_n\vert}^pdx=+\infty$. Without loss of generality, for any $n\geq1$, there exists $M_2>0$ such that $\int_{\mathbb{R}^N}{\vert\nabla u_n\vert}^pdx\geq M_2$. Set
\begin{equation*}
    s_n:=\frac{1}{p_c+p}\ln{\left(\int_{\mathbb{R}^N}{\vert\nabla u_n\vert}^pdx\right)} \text{ and }v_n:=(-s_n)\star u_n.
\end{equation*}
Clearly, $s_n\to+\infty$ as $n\to\infty$, $\{v_n\}\subset S_m$ and
\begin{equation*}
    \int_{\mathbb{R}^N}{\vert\nabla v_n\vert}^pdx=\int_{\mathbb{R}^N}{\vert\nabla((-s_n)\star u_n)\vert}^pdx=(-s_n)^{p_c+p}\int_{\mathbb{R}^N}{\vert\nabla u_n\vert}^pdx=1,
\end{equation*}
\begin{equation*}
          \int_{\mathbb{R}^N}{\vert v_n\vert}^pdx=\int_{\mathbb{R}^N}{\vert(-s_n)\star u_n\vert}^pdx=(-s_n)^{p_c}\int_{\mathbb{R}^N}{\vert u_n\vert}^pdx=\frac{\int_{\mathbb{R}^N}{\vert u_n\vert}^pdx}{\left({\int_{\mathbb{R}^N}{\vert\nabla u_n\vert}^pdx}\right)^{\hat{p}}}\leq\frac{M_1}{{M_2}^{\hat{p}}},
\end{equation*}
where $\hat{p}=1-p/(p_c+p)$. Thus, both ${\Vert v_n\Vert}_{L^p(\mathbb{R}^N)}$ and ${\Vert\nabla v_n\Vert}_{L^p(\mathbb{R}^N)}$ are bounded for $n\geq1$.

\textbf{Case 2}: For any $n\geq1$, there exists a large number $M_1>0$ such that $\int_{\mathbb{R}^N}{\vert\nabla u_n\vert}^pdx\leq M_1$. Note that $\lim_{n\to\infty}{\Vert u_n\Vert}_\mathscr{X}=+\infty$, then $\lim_{n\to\infty}\int_{\mathbb{R}^N}{\vert u_n\vert}^pdx=+\infty$. Without loss of generality, for any $n\geq1$, there exists $M_2>0$ such that $\int_{\mathbb{R}^N}{\vert u_n\vert}^pdx\geq M_2$. Note that $\{u_n\}\subset\mathcal{P}_m$, that is, ${\Vert u_n\Vert}_{L^2(\mathbb{R}^N)}=m<+\infty$, then for $p\neq2$, we can set
\begin{equation*}
    s_n:=\frac{1}{p_c}\ln{\left(\int_{\mathbb{R}^N}{\vert u_n\vert}^pdx\right)} \text{ and } v_n:=(-s_n)\star u_n.
\end{equation*}
Then $\lim_{n\to\infty}s_n=+\infty$ and $\{v_n\}\subset S_m$. Moreover, $\int_{\mathbb{R}^N}{\vert v_n\vert}^pdx=(-s_n)^{p_c}\int_{\mathbb{R}^N}{\vert u_n\vert}^pdx=1$ and
\begin{equation*}
    \int_{\mathbb{R}^N}{\vert\nabla v_n\vert}^pdx=(-s_n)^{p_c+p}\int_{\mathbb{R}^N}{\vert\nabla u_n\vert}^pdx=\frac{\int_{\mathbb{R}^N}{\vert\nabla u_n\vert}^pdx}{\left({\int_{\mathbb{R}^N}{\vert u_n\vert}^pdx}\right)^{1+\frac{p}{p_c}}}\leq\frac{M_1}{{M_2}^{1+\frac{p}{p_c}}}.
\end{equation*}
Therefore, both ${\Vert v_n\Vert}_{L^p(\mathbb{R}^N)}$ and ${\Vert\nabla v_n\Vert}_{L^p(\mathbb{R}^N)}$ are bounded for $n\geq1$.

\textbf{Case 3}: For any $n\geq1$, we have
\begin{equation*}
    \int_{\mathbb{R}^N}{\vert\nabla u_n\vert}^pdx\to+\infty \text{ and } \int_{\mathbb{R}^N}{\vert u_n\vert}^pdx\to+\infty\text{ as }n\to\infty.
\end{equation*}
Without loss of generality, we can set
\begin{equation*}
    s_n:=\min\left\{\frac{1}{p_c+p}\ln{\left(\int_{\mathbb{R}^N}{\vert\nabla u_n\vert}^pdx\right)}, \frac{1}{p_c}\ln{\left(\int_{\mathbb{R}^N}{\vert u_n\vert}^pdx\right)} \right\}.
\end{equation*}
Clearly, $s_n\to+\infty$ as $n\to\infty$, $\{v_n\}\subset S_m$ and
\begin{equation*}
    \int_{\mathbb{R}^N}{\vert\nabla v_n\vert}^pdx=(-s_n)^{p_c+p}\int_{\mathbb{R}^N}{\vert\nabla u_n\vert}^pdx=\frac{\int_{\mathbb{R}^N}{\vert\nabla u_n\vert}^pdx}{\max{\left\{\int_{\mathbb{R}^N}{\vert\nabla u_n\vert}^pdx, \left(\int_{\mathbb{R}^N}{\vert u_n\vert}^pdx \right)^{1+\frac{p}{p_c}}\right\}}},
\end{equation*}
\begin{equation*}
    \int_{\mathbb{R}^N}{\vert v_n\vert}^pdx=(-s_n)^{p_c}\int_{\mathbb{R}^N}{\vert u_n\vert}^pdx=\frac{\int_{\mathbb{R}^N}{\vert u_n\vert}^pdx}{\max{\left\{\left(\int_{\mathbb{R}^N}{\vert\nabla u_n\vert}^pdx \right)^{1-\frac{p}{p_c+p}}, \int_{\mathbb{R}^N}{\vert u_n\vert}^pdx\right\}}}.
\end{equation*}
Hence, we have
\begin{equation*}
    1=1+0\leq\int_{\mathbb{R}^N}{\vert\nabla v_n\vert}^pdx+\int_{\mathbb{R}^N}{\vert v_n\vert}^pdx\leq1+1=2\text{ for any } n\geq1,
\end{equation*}
which implies that both ${\Vert v_n\Vert}_{L^p(\mathbb{R}^N)}$ and ${\Vert\nabla v_n\Vert}_{L^p(\mathbb{R}^N)}$ are bounded for $n\geq1$.

To sum up, there exist two large positive constants $M_3$ and $M_4$ such that $\int_{\mathbb{R}^N}{\vert\nabla v_n\vert}^pdx\leq M_3$ and $\int_{\mathbb{R}^N}{\vert v_n\vert}^pdx\leq M_4$ for all $n\geq1$. Set
\begin{equation*}
    \rho:=\limsup\limits_{n\to+\infty}\left(\sup\limits_{y\in\mathbb{R}^N}\int_{B(y,1)}{\vert v_n\vert}^pdx\right)\ge0.
\end{equation*}

If $\rho>0$, up to a subsequence, there exists $\{y_n\}\subset\mathbb{R}^N$ and $w\in \mathscr{X}\setminus\{0\}$ such that
\begin{equation*}
    w_n:=v_n(\cdot+y_n)\rightharpoonup w \text{ in } \mathscr{X}\text{ and }w_n\to w \text{ a.e. in } \mathbb{R}^N.
\end{equation*}
From Lemma \ref{lem2.2}, Fatou's lemma and the fact that $\lim_{n\to\infty}s_n=+\infty$, we have
\begin{equation*}
     \lim\limits_{n\to+\infty}\int_{\mathbb{R}^N}\Gamma_0(e^{\frac{Ns}{2}}w_n){\vert w_n\vert}^{p_*}dx=+\infty,
\end{equation*}
where $\Gamma_0$ is defined in \eqref{eq2.5} with $c=0$.
Thanks to conclusion (iii) and \eqref{eq2.6} with $c=0$, we further have
\begin{equation*}
     \begin{aligned}
          0 & \leq(-s_n)^{p_c+p}I(u_n)=(-s_n)^{p_c+p}I(s_n\star v_n)\\
          & =\frac{1}{p}\int_{\mathbb{R}^N}{\vert\nabla v_n\vert}^pdx+\frac{1}{p}e^{-ps_n}\int_{\mathbb{R}^N}{\vert v_n\vert}^pdx-\int_{\mathbb{R}^N}\Gamma_0(e^{\frac{Ns_n}{2}}v_n){\vert v_n\vert}^{p_*}dx\\
          & \leq\frac{M_3}{p}+\frac{M_4}{p}e^{-ps_n}-\int_{\mathbb{R}^N}\Gamma_0(e^{\frac{Ns_n}{2}}w_n){\vert w_n\vert}^{p_*}dx\to-\infty\text{ as }n\to\infty,
     \end{aligned}
\end{equation*}
which is a contradiction.

We next consider the case $\rho=0$. Thanks to Lemma \ref{Lemma I.1}, $v_n\to0$ in $L^{p_*}(\mathbb{R}^N)$ as $n\to\infty$. Lemma \ref{lem2.1}(ii) implies that
\begin{equation*}
     \lim\limits_{n\to+\infty}e^{-Ns}\int_{\mathbb{R}^N}F(e^{\frac{Ns}{2}}v_n)dx=0 \text{ for any } s\in\mathbb{R}.
\end{equation*}

In case 1, it follows from $P(s_n\star v_n)=P(u_n)=0$ and Lemma \ref{lem2.4} that there exists $C>0$ large enough such that
\begin{equation*}
     \begin{aligned}
          C&\geq I(u_n)=I(s_n\star u_n) \geq I(s\star v_n)\\
          & =\frac{1}{p}e^{sp_c}\left(e^{sp}\int_{\mathbb{R}^N}{\vert\nabla v_n\vert}^pdx+\int_{\mathbb{R}^N}{\vert v_n\vert}^pdx\right)-e^{-sN}\int_{\mathbb{R}^N}F(e^{\frac{Ns}{2}}v_n)dx\\
          & \geq\frac{1}{p}e^{s(p_c+p)}\int_{\mathbb{R}^N}{\vert\nabla v_n\vert}^pdx-e^{-sN}\int_{\mathbb{R}^N}F(e^{\frac{Ns}{2}}v_n)dx\\
          & =\frac{1}{p}e^{s(p_c+p)}+o_n(1)
     \end{aligned}
\end{equation*}
for all $s\in\mathbb{R}$. However, $\frac{1}{p}e^{s(p_c+p)}>C$ for $s=\frac{1}{p_c+p}\ln{\left(Cp\right)}+1$. We arrive at a contradiction.

Case 2 implies a contradiction in a similar way. Indeed, it follows from $P(s_n\star v_n)=P(u_n)=0$ and Lemma \ref{lem2.4} that there exists $C>0$ large enough such that
\begin{equation*}
     \begin{aligned}
          C&\geq I(u_n)=I(s_n\star u_n)\geq I(s\star v_n)\\
          & \geq\frac{1}{p}e^{sp_c}\int_{\mathbb{R}^N}{\vert v_n\vert}^pdx-e^{-Ns}\int_{\mathbb{R}^N}F(e^{\frac{Ns}{2}}v_n)dx=\frac{1}{p}e^{sp_c}+o_n(1).
     \end{aligned}
\end{equation*}
for all $s\in\mathbb{R}$.  However, $\frac{1}{p}e^{sp_c}>C$ for $s=\frac{1}{p_c}\ln{\left(Cp\right)}+1$. We arrive at a contradiction.

Similarly, case 3 also implies a contradiction. Indeed, it follows from $P(s_n\star v_n)=P(u_n)=0$ and Lemma \ref{lem2.4} that there exists $C>0$ large enough such that
\begin{equation*}
     \begin{aligned}
          C&\geq I(u_n)=I(s_n\star u_n)\geq I(s\star v_n)\\
          & \geq\frac{1}{p}e^{sp_c}\min\left\{e^{sp}, 1\right\}{\Vert v_n\Vert}^p_{W^{1,p}(\mathbb{R}^N)}-e^{-Ns}\int_{\mathbb{R}^N}F(e^{\frac{Ns}{2}}v_n)dx\\
          & \geq\frac{1}{p}e^{sp_c}\min\left\{e^{sp}, 1\right\}-e^{-Ns}\int_{\mathbb{R}^N}F(e^{\frac{Ns}{2}}v_n)dx\\
          & =\frac{1}{p}e^{sp_c}\min\left\{e^{sp}, 1\right\}+o_n(1)
     \end{aligned}
\end{equation*}
for all $s\in\mathbb{R}$. However, $\frac{1}{p}e^{sp_c}>C$ for $s=\frac{1}{p_c}\ln{\left(Cp\right)}+1$. We arrive at a contradiction.
 $\hfill\square$
\\ \hspace*{\fill}
\begin{remark}\label{re2.0}
Except for Cases 1,2,3 discussed in the proof of Theorem \ref{lem2.5}, there are two other possible cases satisfying $\lim_{n\to\infty}{\Vert u_n\Vert}_\mathscr{X}=+\infty$ and $\sup_{n\geq1}I(u_n)\leq C$.

\textbf{Case 4}: Only one of ${\Vert u_n\Vert}_{L^p(\mathbb{R}^N)}$ and ${\Vert\nabla u_n\Vert}_{L^p(\mathbb{R}^N)}$ goes to infinity as $n\to\infty$. Moreover, ${\Vert\nabla u_n\Vert}_{L^p(\mathbb{R}^N)}$ is unbounded when $\lim_{n\to\infty}{\Vert u_n\Vert}_{L^p(\mathbb{R}^N)}=+\infty$ and ${\Vert u_n\Vert}_{L^p(\mathbb{R}^N)}$ is unbounded when $\lim_{n\to\infty}{\Vert\nabla u_n\Vert}_{L^p(\mathbb{R}^N)}=+\infty$;

\textbf{Case 5}: Neither ${\Vert u_n\Vert}_{L^p(\mathbb{R}^N)}$ nor ${\Vert\nabla u_n\Vert}_{L^p(\mathbb{R}^N)}$ goes to infinity as $n\to\infty$. Moreover, both ${\Vert u_n\Vert}_{L^p(\mathbb{R}^N)}$ and ${\Vert\nabla u_n\Vert}_{L^p(\mathbb{R}^N)}$ are unbounded.

Without loss of generality, assume $\lim_{n\to\infty}{\Vert\nabla  u_n\Vert}_{L^p(\mathbb{R}^N)}=+\infty$ and ${\Vert u_n\Vert}_{L^p(\mathbb{R}^N)}$ is unbounded in case 4. Obviously, there exists a subsequence $\{u_{n_k}\}_{k=1}^{\infty}$ of $\{u_n\}$ such that $\lim_{k\to\infty}{\Vert u_{n_k}\Vert}_{L^p(\mathbb{R}^N)}=+\infty$. Repeating the proof of case 3, one can easily obtain the desired result.

Here is an example for case 5. For any $k\geq1$, there exists $C_*>0$ such that
\begin{equation*}
     \left\{
     \begin{aligned}
          &\int_{\mathbb{R}^N}{\vert\nabla u_k\vert}^pdx\to+\infty, \int_{\mathbb{R}^N}{\vert u_k\vert}^pdx\leq C_*,
          &\text{ if } k \text{ is a prime},\\
          &\int_{\mathbb{R}^N}{\vert\nabla u_k\vert}^pdx\leq C_*, \int_{\mathbb{R}^N}{\vert u_k\vert}^pdx\to+\infty,
          &\text{ if } k \text{ is not a prime}.
     \end{aligned}
     \right.
\end{equation*}
Nevertheless, we can assume that
\begin{equation*}
     \left\{
     \begin{aligned}
          &s_k:=\frac{1}{p_c+p}\ln{\left(\int_{\mathbb{R}^N}{\vert\nabla u_k\vert}^pdx\right)},
          &\text{ if } k \text{ is a prime},\\
          &s_k:=\frac{1}{p_c}\ln{\left(\int_{\mathbb{R}^N}{\vert u_k\vert}^pdx\right)},
          &\text{ if } k \text{ is not a prime}.
     \end{aligned}
     \right.
\end{equation*}
Similar to the proof of cases 1 and 2, both ${\Vert v_n\Vert}_{L^p(\mathbb{R}^N)}$ and ${\Vert\nabla v_n\Vert}_{L^p(\mathbb{R}^N)}$ are bounded.
\end{remark}

\begin{remark}\label{re2.3}
Assume $N\geq3$ and $f$ satisfies $(f0)-(f4)$. Then for each sequence $\{u_n\}\subset \mathscr{X}\setminus\{0\}$ such that $ P(u_n)=0,~\sup_{n\geq1}{\Vert u_n\Vert}_{L^2(\mathbb{R}^N)}<+\infty$ and $\sup_{n\geq1}I(u_n)<+\infty$,
repeating the proof of Lemma \ref{lem2.5}(iv), we can deduce that $\{u_n\}$ is bounded in $\mathscr{X}$.
\end{remark}

\section{The behavior of the function $m\mapsto E_m$}
When $N\geq3$ and $f$ satisfies $(f0)-(f4)$, it follows from Lemma \ref{lem2.5} that for any given $m>0$, following infimum is well-defined and strictly positive:
\begin{equation*}
    E_m:=\inf\limits_{m\in\mathcal{P}_m}I(u).
\end{equation*}
In this section, we mainly discuss the characteristic behavior of $E_m$ when $m>0$, especially prove that $E_m$ is nonincreasing for $m>0$. Firstly, we have following continuity result of $E_m$.

\begin{lemma}\label{lem3.1}
Assume $N\geq3$ and $f$ satisfies $(f0)-(f4)$, then the function $m\mapsto E_m$ is continuous at each $m>0$.
\end{lemma}
\begin{proof}
It is enough to prove that for any given $m>0$ and positive sequence $\{m_k\}$ such that $\lim_{k\to\infty}m_k=m$, we have $\lim_{k\to+\infty}E_{m_k}=E_m$.

We first prove that
\begin{equation}\label{eq3.1}
    \limsup\limits_{k\to+\infty}E_{m_k}\leq E_m.
\end{equation}
For any $u\in\mathcal{P}_m$ and $ k\in\mathbb{N}^+$, define
\begin{equation*}
    u_k:=\sqrt{\frac{m_k}{m}}u\in S_{m_k}.
\end{equation*}
Note that $\lim_{k\to+\infty}u_k=u$ in $\mathscr{X}$. It follows from Lemma \ref{lem2.4}(iii) that
$\lim_{k\to+\infty}s(u_k)=s(u)=0$, which leads to
\begin{equation*}
    \lim_{k\to+\infty}s(u_k)\star u_k=s(u)\star u=u\text{ in } \mathscr{X}.
\end{equation*}
Thus, we have
\begin{equation*}
    \limsup\limits_{k\to+\infty}E_{m_k}\leq\limsup\limits_{k\to+\infty}I(s(u_k)\star u_k)=I(u).
\end{equation*}
By the arbitrary of $u$, \eqref{eq3.1} clearly holds.

Next, we prove the following opposite inequality
\begin{equation}\label{eq3.2}
    \liminf\limits_{k\to+\infty}E_{m_k}\geq E_m.
\end{equation}
Set
\begin{equation*}
    t_k:={\left(\frac{m}{m^k}\right)}^{\frac{1}{N}} \quad \text{ and } \quad \tilde{v}_k(\cdot):=v_k\left(\frac{\cdot}{t_k}\right)\in S_m.
\end{equation*}
For each $k\in\mathbb{N}^+$, there exists $v_k\in\mathcal{P}_{m_k}$ such that
\begin{equation}\label{eq3.3}
I(v_k)\leq E_{m_k}+1/k.
\end{equation}
Combining \eqref{eq3.3} with Lemma \ref{lem2.4}(ii), we have
\begin{equation*}
     \begin{aligned}
          E_m &\leq I(s(\tilde{v}_k)\star\tilde{v}_k)\leq I(s(\tilde{v}_k\star v_k)+\left\vert I(s(\tilde{v}_k)\star\tilde{v}_k)-I(s(\tilde{v}_k\star v_k)\right\vert\\
          & \leq I(v_k)+\left\vert I(s(\tilde{v}_k)\star\tilde{v}_k)-I(s(\tilde{v}_k\star v_k)\right\vert\\
          & \leq E_{m_k}+\frac{1}{k}+\left\vert I(s(\tilde{v}_k)\star\tilde{v}_k)-I(s(\tilde{v}_k\star v_k)\right\vert\\
          & =:E_{m_k}+\frac{1}{k}+\zeta(k)
     \end{aligned}
\end{equation*}
Obviously, \eqref{eq3.2} holds of and only if
\begin{equation}\label{eq3.4}
    \lim\limits_{k\to+\infty}\zeta(k)=0.
\end{equation}
Since $s\star\left(v(\cdot/t)\right)=(s\star v)(\cdot/t)$, we have
\begin{equation*}
     \begin{aligned}
          \zeta(k) &=\left\vert\frac{1}{p}\left(t^{N-p}_k-1\right)\int_{\mathbb{R}^N}{\vert\nabla \left(s(\tilde{v}_k)\star v_k\right)\vert}^p+{\vert \left(s(\tilde{v}_k)\star v_k\right)\vert}^pdx-\left(t_k^N-1\right)\int_{\mathbb{R}^N}F\left(s(\tilde{v}_k)\star v_k\right)dx\right\vert\\
          & \leq\frac{1}{p}\left\vert t^{N-p}_k-1\right\vert\int_{\mathbb{R}^N}{\vert\nabla \left(s(\tilde{v}_k)\star v_k\right)\vert}^p+{\vert \left(s(\tilde{v}_k)\star v_k\right)\vert}^pdx+\left\vert t_k^N-1\right\vert\int_{\mathbb{R}^N}\left\vert F\left(s(\tilde{v}_k)\star v_k\right)\right\vert dx\\
          & =:\frac{1}{p}\left\vert t^{N-p}_k-1\right\vert\varphi(k)+\left\vert t_k^N-1\right\vert\psi(k).
     \end{aligned}
\end{equation*}
Since $\lim_{k\to\infty}t_k=1$, the proof of \eqref{eq3.2} is reduced to
\begin{equation}\label{eq3.5}
    \limsup\limits_{k\to+\infty}\varphi(k)<+\infty\text{ and } \limsup\limits_{k\to+\infty}\psi(k)<+\infty.
\end{equation}
In what follows, we divide the proof of \eqref{eq3.5} into three claims.

\textbf{Claim 1.} The sequence $\{v_k\}$ is bounded in $\mathscr{X}$.

Indeed, \eqref{eq3.1} and \eqref{eq3.3} imply that $\limsup_{k\to+\infty}I(v_k)\leq E_m$. Note that $\lim_{k\to+\infty}m_k=m$ and $v_k\in\mathcal{P}_{m_k}$. Hence, it follows from Remark \ref{re2.3} that Claim 1 holds.

\textbf{Claim 2.} The sequence $\{\tilde{v}_k\}$ is bounded in $\mathscr{X}$. Moreover, up to a subsequence, there exists $\{y_k\}\subset\mathbb{R}^N$ and $v\in\mathscr{X}$ such that $\tilde{v}_k(\cdot+y_k)\to v$ a.e. in $\mathbb{R}^N$ as $k\to+\infty$ and $v\neq0$.

Since $\lim_{k\to+\infty}t_k=1$, Claim 1 implies that $\{\tilde{v}_k\}$ is bounded in $\mathscr{X}$. Set
\begin{equation*}
    \rho:=\limsup\limits_{k\to+\infty}\Big(\sup\limits_{y\in\mathbb{R}^N}\int_{B(y,1)}{\vert\tilde{v}_k\vert}^pdx\Big).
\end{equation*}
It only remains to prove that $\rho\neq0$. If $\rho=0$, then it follows from Lemma \ref{Lemma I.1} that $\tilde{v}_k\to0$ in $L^{p_*}(\mathbb{R}^N)$ as $k\to+\infty$, which leads to
\begin{equation*}
    \int_{\mathbb{R}^N}{\vert v_k\vert}^{p_*}dx=\int_{\mathbb{R}^N}{\vert \tilde{v}_k(\cdot\  t_k)\vert}^{p_*}dx=t^{-N}_k\int_{\mathbb{R}^N}{\vert \tilde{v}_k\vert}^{p_*}dx\to0\text{ as }k\to+\infty.
\end{equation*}
Thanks to Lemma \ref{lem2.1}(ii) and $P(v_k)=0$, we have
\begin{equation*}
    \frac{p_c+p}{p}\int_{\mathbb{R}^N}{\vert\nabla v_k\vert}^pdx+\frac{p_c}{p}\int_{\mathbb{R}^N}{\vert v_k\vert}^pdx=\frac{N}{2}\int_{\mathbb{R}^N}\widetilde{F}(v_k)dx\to0\text{ as }k\to+\infty,
\end{equation*}
which implies $\lim_{k\to+\infty}\int_{\mathbb{R}^N}{\vert\nabla v_k\vert}^pdx=0$. In view of Remark \ref{re2.1}, it holds
\begin{equation*}
    0=P(v_k)\geq\frac{p_c+p}{p}\int_{\mathbb{R}^N}{\vert\nabla v_k\vert}^pdx+\frac{p_c}{p}\int_{\mathbb{R}^N}{\vert v_k\vert}^pdx>0\text{ for } k \text{ large enough}.
\end{equation*}
We arrive at a contradiction and hence Claim 2 holds.

\textbf{Claim 3.} $\limsup_{k\to+\infty}s(\tilde{v}_k)<+\infty$.

Suppose there exists a subsequence of $\tilde{v}_k$, still denoted it by $\tilde{v}_k$, such that
\begin{equation}\label{eq3.6}
    s(\tilde{v}_k)\to+\infty \quad \text{ as } k\to+\infty.
\end{equation}
By Claim 2, we have
\begin{equation}\label{eq3.7}
    \tilde{v}_k(\cdot+y_k)\to v\neq0 \quad \text{ a.e. in } \mathbb{R}^N.
\end{equation}
On the other hand, it follows from Lemma \ref{lem2.4}(iv) and \eqref{eq3.6} that
\begin{equation}\label{eq3.8}
    s(\tilde{v}_k(\cdot+y_k))=s(\tilde{v}_k)\to+\infty \text{ as } k\to+\infty.
\end{equation}
Moreover, Lemma \ref{lem2.4}(ii) implies that
\begin{equation}\label{eq3.9}
    I(s(\tilde{v}_k(\cdot+y_k))\star\tilde{v}_k(\cdot+y_k))\geq0.
\end{equation}
Thanks to \eqref{eq3.7}$-$\eqref{eq3.9}, we can get a contradiction in the same way as derivation of \eqref{eq2.12}. Then the Claim 3 is proved.

From Claims 1 and 3, we have $\limsup_{k\to+\infty}{\Vert s(\tilde{v}_k)\star v_k\Vert}_\mathscr{X}<+\infty$. Since $f$ satisfies $(f0)-(f4)$, it is clear that \eqref{eq3.5} holds. We complete the proof.
\end{proof}

\begin{lemma}\label{lem3.2}
Assume $N\geq3$ and $f$ satisfies $(f0)-(f4)$, then the function $m\mapsto E_m$ is nonincreasing on $(0,+\infty)$.
\end{lemma}
\begin{proof}
We only need to prove that for any $m_1>m_2>0$ and any constant $\varepsilon>0$, it holds
\begin{equation}\label{eq3.10}
    E_{m_1}\leq E_{m_2}+\varepsilon.
\end{equation}
By the definition of $E_{m_2}$, there exists $u\in\mathcal{P}_{m_2}$ such that
\begin{equation}\label{eq3.11}
    I(u)\leq E_{m_2}+\frac{\varepsilon}{2}.
\end{equation}
Let $\chi\in C^\infty_0(\mathbb{R}^N)$ be a radial function such that $0\leq\chi(x)\leq1$ in $\mathbb{R}^N$ and
\begin{equation*}
     \chi(x)=\left\{
     \begin{aligned}
        &1, \quad \text{ if } \vert x\vert\leq1,\\
          &0, \quad \text{ if } \vert x\vert\geq2.
     \end{aligned}
     \right.
\end{equation*}
For any $\delta>0$ small, we define $u_\delta(x)=u(x)\cdot\chi(\delta x)\in\mathscr{X}\backslash\{0\}$. Since $u_\delta\to u$ in $\mathscr{X}$ as $\delta\to0^+$, by Lemma \ref{lem2.4} (iii), we have $\lim_{\delta\to0^+}s(u_\delta)=s(u)=0$, which leads to
\begin{equation*}
    s(u_\delta)\star u_\delta\to s(u)\star u=u\text{ in } \mathscr{X} \text{ as } \delta\to0^+.
\end{equation*}
We can choose $\delta>0$ small enough such that
\begin{equation}\label{eq3.12}
    I(s(u_\delta)\star u_\delta)\leq I(u)+\frac{\varepsilon}{4}.
\end{equation}
Choose $v\in C^\infty_0(\mathbb{R}^N)$ such that spt$(v)\subset B(0,1+4/\delta)\backslash B(0,4/\delta)$ and set
\begin{equation*}
    \tilde{v}=\frac{m_1-{\Vert u_\delta\Vert}^2_{L^2(\mathbb{R}^N)}}{{\Vert v\Vert}^2_{L^2(\mathbb{R}^N)}}v.
\end{equation*}
For any $\sigma\leq0$, we define $w_\sigma:=u_\delta+\sigma\star\tilde{v}$. A direct computation yields that
\begin{equation*}
    \text{spt}(u_\delta)\cap\text{spt}(\sigma\star\tilde{v})=\emptyset,
\end{equation*}
and thus $w_\sigma\in S_m$.

We claim that $s(w_\sigma)$ is bounded from above when $\sigma\to-\infty$. Indeed, it follows from Lemma \ref{lem2.4}(ii) that $I(s(w_\sigma)\star w_\sigma)\geq0$ and $w_\sigma\to u_\delta$ a.e. in $\mathbb{R}^N$ as $\sigma\to-\infty$ with $u_\delta\neq0$. We can infer to a contradiction in the same way to the derivation of \eqref{eq2.12} if the claim is not true.

Since $\lim_{\sigma\to-\infty}(s(w_\sigma)+\sigma)=-\infty$, then as $\sigma\to-\infty$, we have
\begin{equation*}
    {\Vert \left(s(w_\sigma)+\sigma\right)\star\tilde{v}\Vert}_{L^p(\mathbb{R}^N)}\to0, \quad {\Vert\nabla\left[\left(s(w_\sigma)+\sigma\right)\star\tilde{v}\right]\Vert}_{L^p(\mathbb{R}^N)}\to0
\end{equation*}
and
\begin{equation*}
    \left(s(w_\sigma)+\sigma\right)\star\tilde{v}\to0\text{ in } L^{p_*}(\mathbb{R}^N).
\end{equation*}
Using Lemma \ref{lem2.1} (ii), we have
\begin{equation}\label{eq3.13}
    I\left((s(w_\sigma)+\sigma\right)\star\tilde{v})\leq\frac{\varepsilon}{4} \text{ for } \sigma<0 \text{ small enough}.
\end{equation}
Thanks to Lemma \ref{lem2.1}(ii) and \eqref{eq3.11}$-$\eqref{eq3.13}, we further have
\begin{equation*}
     \begin{aligned}
         E_{m_1}&\leq I(s(w_\sigma)\star w_\sigma)=I(s(w_\sigma)\star u_\delta)+I(s(w_\sigma)\star(\sigma\star\tilde{v}))\\
         & \leq I(s(u_\delta)\star u_\delta)+I\left((s(w_\sigma)+\sigma\right)\star\tilde{v})\\
         & \leq I(u)+\frac{\varepsilon}{2}\leq E_{m_2}+\varepsilon.
     \end{aligned}
\end{equation*}
This completes the proof.
\end{proof}

\begin{lemma}\label{lem3.3}
Assume $N\geq3$ and $f$ satisfies $(f0)-(f4)$. If there exist $\mu\in\mathbb{R}$ and $u\in S_m$ such that $I(u)=E_m$ and
\begin{equation*}
    -\Delta_{p}u+{\vert u\vert}^{p-2}u=f(u)-\mu u,
\end{equation*}
then for any $m'>m$ close to $m$ when $\mu>0$ or for any $m'<m$ close to $m$ when $\mu<0$, we have  $E_m>E_{m'}$.
\end{lemma}

\begin{proof}
For any $t>0$ and $s\in\mathbb{R}$, set $u_{t,s}:=s\star(tu)\in S_{mt^2}$. Then
\begin{equation*}
    \Upsilon(t,s):=I(u_{t,s})=\frac{1}{p}t^p(e^s)^{p_c+p}\int_{\mathbb{R}^N}{\vert\nabla u\vert}^pdx+\frac{1}{p}t^pe^{sp_c}\int_{\mathbb{R}^N}{\vert u\vert}^pdx-e^{-Ns}\int_{\mathbb{R}^N}F(te^{\frac{Ns}{2}}u)dx,
\end{equation*}
Some simple calculations yield that
\begin{equation*}
     \begin{aligned}
         \frac{\partial}{\partial t}\Upsilon(t,s) &=t^{p-1}(e^s)^{p_c+p}\int_{\mathbb{R}^N}{\vert\nabla u\vert}^pdx+t^{p-1}e^{sp_c}\int_{\mathbb{R}^N}{\vert u\vert}^pdx-e^{-Ns}\int_{\mathbb{R}^N}f(te^{\frac{Ns}{2}}u)e^{\frac{Ns}{2}}udx\\
         & =\frac{1}{t}I'(u_{t,s})u_{t,s}.
     \end{aligned}
\end{equation*}
For $\mu>0$, we have $\lim_{(t,s)\to(1,0)}u_{t,s}=u$ in $\mathscr{X}$ and
$I'(u)u=-\mu{\Vert u\Vert}^2_{L^2(\mathbb{R}^N)}=-\mu m<0$. Choose $\delta>0$ small enough such that
\begin{equation*}
    \frac{\partial}{\partial t}\Upsilon(t,s)<0 \text{ for any } (t,s)\in(1,1+\delta]\times[-\delta,\delta].
\end{equation*}
Using the mean value theorem, for any $1<\theta<t\leq1+\delta$ and $\vert s\vert\leq\delta$, we have
\begin{equation}\label{eq3.14}
    \Upsilon(t,s)=\Upsilon(1,s)+(t-1)\frac{\partial}{\partial t}\Upsilon(\theta,s)<\Upsilon(1,s).
\end{equation}
Lemma \ref{lem2.4}(iii) implies that $\lim_{t\to1^+}s(tu)=s(u)=0$. Thus, for any $m'>m$ close enough to $m$, we have
\begin{equation*}
t:=\sqrt{\frac{m'}{m}}\in(1,1+\delta]\text{ and }s:=s(tu)\in[-\delta,\delta].
\end{equation*}
The case $\mu<0$ can be proved in a similar way, we omit the details here.
\end{proof}

Following conclusion is a direct result of Lemmas \ref{lem3.2} and \ref{lem3.3}.

\begin{lemma}\label{lem3.4}
Assume $N\geq3$ and $f$ satisfies $(f0)-(f4)$. If there exists $\mu\in\mathbb{R}$ and $u\in S_m$ such that $I(u)=E_m$ and
\begin{equation*}
-\Delta_{p}u+{\vert u\vert}^{p-2}u=f(u)-\mu u,
\end{equation*}
then $\mu\geq0$. Moreover, $E_m>E_{m'}$ for any $m'>m$ when $\mu>0$.
\end{lemma}

At the end of this section, we give following limit behavior of $E_m$ when $m\to0^+$.

\begin{lemma}\label{lem3.5}
Assume $N\geq3$ and $f$ satisfies $(f0)-(f4)$. Then $\lim_{m\to0^+}E_m=+\infty$.
\end{lemma}

\begin{proof}
 It is enough to show that for any sequence $\{u_n\}\subset\mathscr{X}\backslash\{0\}$ such that $ P(u_n)=0$ and $\lim_{n\to+\infty}{\Vert u_n\Vert}_{L^2(\mathbb{R}^N)}=0$,
we have $\lim_{n\to+\infty}I(u_n)=\infty$. Thanks to Lemma \ref{lem2.5}(iv), we can define
\begin{equation*}
s_n:=\min\left\{\frac{1}{p_c+p}\ln{\left(\int_{\mathbb{R}^N}{\vert\nabla u_n\vert}^pdx\right)}, \frac{1}{p_c}\ln{\left(\int_{\mathbb{R}^N}{\vert u_n\vert}^pdx\right)} \right\}\text{ and }v_n:=(-s_n)\star u_n.
\end{equation*}
Some direct computations yield that
$$1\leq{\Vert v_n\Vert}^p_{W^{1,p}(\mathbb{R}^N)}\leq2\text{ and }\lim\limits_{n\to\infty}{\Vert v_n\Vert}_{L^2(\mathbb{R}^N)}=\lim\limits_{n\to\infty}{\Vert u_n\Vert}_{L^2(\mathbb{R}^N)}=0.$$
It follows from Lemma \ref{Lemma I.1} and Lemma \ref{lem2.1}(ii)that $\lim_{n\to\infty}v_n=0$ in $L^{p_*}(\mathbb{R}^N)$ and
\begin{equation*}
\lim\limits_{n\to+\infty}e^{-Ns}\int_{\mathbb{R}^N}F(e^{\frac{Ns}{2}}v_n)dx=0\text{ for any } s\in\mathbb{R}.
\end{equation*}
Thanks to Lemma \ref{lem2.4} and the fact that $P(s_n\star v_n)=P(u_n)=0$, we have
\begin{equation*}
     \begin{aligned}
         I(u_n)&=I(s_n\star v_n)\geq I(s\star v_n)\\
         & =\frac{1}{p}(e^s)^{p_c+p}\int_{\mathbb{R}^N}{\vert\nabla v_n\vert}^pdx+\frac{1}{p}e^{sp_c}\int_{\mathbb{R}^N}{\vert v_n\vert}^pdx-e^{-Ns}\int_{\mathbb{R}^N}F(e^{\frac{Ns}{2}}v_n)dx\\
         & \geq\frac{1}{p}e^{sp_c}+o_n(1).
     \end{aligned}
\end{equation*}
By the arbitrariness of $s\in\mathbb{R}$, we have $\lim_{n\to\infty}I(u_n)=+\infty$. This complete the proof.
\end{proof}
\section{Ground states}
This section is devoted to the proof of Theorems \ref{theo1.1} and \ref{theo1.2} by establishing the existence of ground states to \eqref{(P_m)} and completing the properties of the ground state energy function $E_m$. We first focus on the proof of Theorem \ref{theo1.1} and give the following lemmas.

\begin{lemma}\label{lem4.1}
For the constrained functional $I|_{S_m}$ at the level $E_m$, there exists a Palais-Smale sequence $\{u_n\}\subset\mathcal{P}_m$. Moreover, if $f$ is odd, we have in addition that ${\Vert u^-_n\Vert}_{L^2(\mathbb{R}^N)}\to0$, where $u_n^-$ represent the negative part of $u_n$.
\end{lemma}

In order to prove Lemma \ref{lem4.1}, we refer to some arguments from \cite{JeanjeanCVPDE2020,BartschJFA2018,BartschCVPDE2019}. Firstly, we introduce the following functional $\Phi:\mathscr{X}\backslash\{0\}\to\mathbb{R}$ defined by
\begin{equation*}
     \begin{aligned}
          \Phi(u) &:=I(s(u)\star u)\\
          & =\frac{1}{p}\left(e^{s(u)}\right)^{p_c+p}\int_{\mathbb{R}^N}{\vert\nabla u\vert}^pdx+\frac{1}{p}\left(e^{s(u)}\right)^{p_c}\int_{\mathbb{R}^N}{\vert u\vert}^pdx-e^{-Ns(u)}\int_{\mathbb{R}^N}F(e^{\frac{Ns(u)}{2}}u)dx
     \end{aligned}
\end{equation*}
where the unique number $s(u)\in\mathbb{R}$ is guaranteed by Lemma \ref{lem2.4}.

\begin{lemma}\label{lem4.2}
The functional $\Phi:\mathscr{X}\backslash\{0\}\to\mathbb{R}$ is $C^1$ and
\begin{equation*}
     \begin{aligned}
          d\Phi(u)[\varphi] &=\left(e^{s(u)}\right)^{p_c+p}\int_{\mathbb{R}^N}{\vert\nabla u\vert}^{p-2}\nabla u\nabla\varphi dx+\left(e^{s(u)}\right)^{p_c}\int_{\mathbb{R}^N}{\vert u\vert}^{p-2}u\varphi dx\\
          & -e^{-Ns(u)}\int_{\mathbb{R}^N}f(e^{\frac{Ns(u)}{2}}u)e^{\frac{Ns(u)}{2}}\varphi dx\\
          & =dI(s(u)\star u)[s(u)\star\varphi]
     \end{aligned}
\end{equation*}
for any $u\in\mathscr{X}\backslash\{0\}$ and $\varphi\in\mathscr{X}$.
\end{lemma}
\begin{proof}
A direct computation shows that
\begin{equation*}
\frac{d}{dt}{\vert\nabla u+t\nabla\varphi\vert}^p=p{\vert\nabla u+t\nabla\varphi\vert}^{p-2}\left(\nabla u\cdot\nabla\varphi+t{\vert\nabla\varphi\vert}^2\right),
\end{equation*}
\begin{equation*}
\frac{d}{dt}{\vert u+t\varphi\vert}^p=p{\vert u+t\varphi\vert}^{p-2}\left(u\varphi+t{\vert\varphi\vert}^2\right).
\end{equation*}
Let $\varphi\in\mathscr{X}, u\in\mathscr{X}\backslash\{0\}$ and define $s_t:=s(u+t\varphi)$. From the definition of $\Phi$, we have
\begin{equation*}
\Phi(u+t\varphi)-\Phi(u)=I(s_t\star(u+t\varphi))-I(s_0\star u)
\end{equation*}
where $\vert t\vert$ is small enough. Since $s_0=s(u)$ is the unique maximum point of the function $I(s\star u)$, by the mean value theorem, we have
\begin{equation*}
     \begin{aligned}
          I(s_t\star(u+t\varphi))-I(s_0\star u) &\leq I(s_t\star(u+t\varphi))-I(s_t\star u)\\
          & =\frac{1}{p}\left(e^{s_t}\right)^{p_c+p}\int_{\mathbb{R}^N}\left[{\vert\nabla u+t\nabla\varphi\vert}^p-{\vert\nabla u\vert}^p\right]dx\\
          & +\frac{1}{p}\left(e^{s_t}\right)^{p_c}\int_{\mathbb{R}^N}\left[{\vert u+t\varphi\vert}^p-{\vert u\vert}^p\right]dx\\
          & -e^{-Ns_t}\int_{\mathbb{R}^N}\left[F(e^{\frac{Ns_t}{2}}(u+t\varphi))-F(e^{\frac{Ns_t}{2}}u)\right]dx\\
          & =\left(e^{s_t}\right)^{p_c+p}\int_{\mathbb{R}^N}{\vert\nabla u+\eta_1t\nabla\varphi\vert}^{p-2}\left(\nabla u\cdot\nabla\varphi+\eta_1t{\vert\nabla\varphi\vert}^2\right)tdx\\
          & +\left(e^{s_t}\right)^{p_c}\int_{\mathbb{R}^N}{\vert u+\eta_2t\varphi\vert}^{p-2}\left(u\varphi+\eta_2t{\vert\varphi\vert}^2\right)tdx\\
          & -e^{-Ns_t}\int_{\mathbb{R}^N}f(e^{\frac{Ns_t}{2}}(u+\eta_3t\varphi))e^{\frac{Ns_t}{2}}t\varphi dx
     \end{aligned}
\end{equation*}
where $\eta_1, \eta_2, \eta_3\in(0,1)$. Similarly, we have
\begin{equation*}
     \begin{aligned}
          I(s_t\star(u+t\varphi))-I(s_0\star u) &\geq I(s_0\star(u+t\varphi))-I(s_0\star u)\\
          & =\left(e^{s_0}\right)^{p_c+p}\int_{\mathbb{R}^N}{\vert\nabla u+\xi_1t\nabla\varphi\vert}^{p-2}\left(\nabla u\cdot\nabla\varphi+\xi_1t{\vert\nabla\varphi\vert}^2\right)tdx\\
          & +\left(e^{s_0}\right)^{p_c}\int_{\mathbb{R}^N}{\vert u+\xi_2t\varphi\vert}^{p-2}\left(u\varphi+\xi_2t{\vert\varphi\vert}^2\right)tdx\\
          & -e^{-Ns_0}\int_{\mathbb{R}^N}f(e^{\frac{Ns_0}{2}}(u+\xi_3t\varphi))e^{\frac{Ns_t}{2}}t\varphi dx
     \end{aligned}
\end{equation*}
where $\xi_1, \xi_2, \xi_3\in(0,1)$. It follows from Lemma \ref{lem2.4}(iii) that $\lim_{t\to0}s_t=s_0=s(u)$. Then using above of two inequalities, we can obtain
\begin{equation*}
     \begin{aligned}
          \lim\limits_{t\to0}\frac{\Phi(u+t\varphi)-\Phi(u)}{t} &=\left(e^{s(u)}\right)^{p_c+p}\int_{\mathbb{R}^N}{\vert\nabla u\vert}^{p-2}\nabla u\cdot\nabla\varphi dx+\left(e^{s(u)}\right)^{p_c}\int_{\mathbb{R}^N}{\vert u\vert}^{p-2}u\varphi dx\\
          & -e^{-Ns(u)}\int_{\mathbb{R}^N}f(e^{\frac{Ns(u)}{2}}u)e^{\frac{Ns(u)}{2}}\varphi dx.
     \end{aligned}
\end{equation*}
Using Lemma \ref{lem2.4}(iii) again, the G\^ateaux derivative of $\Phi$ is continuous in $u$ and linearly bounded in $\varphi$. Hence $\Phi$ is a $C^1$ function \cite{Willem1997,Badiale2010} and by changing variables in the integrals, we have
\begin{equation*}
     \begin{aligned}
          d\Phi(u)[\varphi] &=\int_{\mathbb{R}^N}{\vert\nabla(s(u))\star u\vert}^{p-2}\nabla(s(u)\star u)\cdot\nabla(s(u)\star\varphi) dx\\
          & +\int_{\mathbb{R}^N}{\vert s(u)\star u\vert}^{p-2}(s(u)\star u)(s(u)\star \varphi)dx-\int_{\mathbb{R}^N}f(s(u)\star u)(s(u)\star\varphi)dx\\
          & =dI(s(u)\star u)[s(u)\star\varphi].
     \end{aligned}
\end{equation*}
This completes the proof.
\end{proof}

For any given $m>0$, we consider the following constrained functional
\begin{equation*}
J:=\Phi|_{S_m}:S_m\to\mathbb{R}.
\end{equation*}
Clearly, $J$ has following conclusion.
\begin{lemma}\label{lem4.3}
The functional $J:S_m\to\mathbb{R}$ is $C^1$ and for any $u\in S_m$ and $\varphi\in T_uS_m$,
\begin{equation*}
dJ(u)[\varphi]=d\Phi(u)[\varphi]=dI(s(u)\star u)[s(u)\star\varphi].
\end{equation*}
\end{lemma}

We recall below definition from \cite{Ghoussoub1993} and then establish a result to prove that a minimax value of $J$ will produce a Palais-Smale sequence, which is made of elements of $\mathcal{P}_m$ for the constrained functional $I|_{S_m}$ at the same level.

\begin{definition}\label{def4.4}(see \cite{Ghoussoub1993})
Let $B$ be a closed subset of $\mathscr{X}$. We shall say that a class $\mathcal{F}$ of compact subsets of $\mathscr{X}$ is a homotopy-stable family with closed boundary $B$ provided

(a) every set in $\mathcal{F}$ contains $B$;

(b) for any set $A\in\mathcal{F}$ and any homotopy $\eta\in C([0,1]\times\mathscr{X},\mathscr{X})$ satisfying $\eta(t,u)=u$ for all $(t,u)\in(\{0\}\times X)\cup([0,1]\times B)$ we have that $\eta(\{1\}\times A)\in\mathcal{F}$.
\end{definition}
We remark that the case $B=\emptyset$ is admissible.

\begin{lemma}\label{lem4.5}
Let $\mathcal{F}$ be a homotopy-stable family of compact subsets of $S_m$ with $B=\emptyset$ and set
\begin{equation*}
E_{m,\mathcal{F}}:=\inf\limits_{A\in\mathcal{F}}\max\limits_{u\in A}J(u).
\end{equation*}
If $E_{m,\mathcal{F}}>0$, then for the constrained functional $I|_{S_m}$, there exists a Palais-Smale sequence $\{u_n\}\subset\mathcal{P}_m$ at the level $E_{m,\mathcal{F}}$. In particular, if $f$ is odd and $\mathcal{F}$ is the class of all set-with-one-element included in $S_m$, then in addition we have that ${\Vert u^-_n\Vert}_{L^2(\mathbb{R}^N)}\to0$.
\end{lemma}

\begin{proof}
Let $\{A_n\}\subset\mathcal{F}$ be a minimazing sequence of $E_{m,\mathcal{F}}$. Thanks to Lemma \ref{lem2.4}(iii), following continuous function is well defined
\begin{equation*}
\eta:[0,1]\times S_m\to S_m, \quad \eta(t,u)=(ts(u))\star u,
\end{equation*}
which satisfies $\eta(t,u)=u$ for all $(t,u)\in\{0\}\times S_m$. Therefore, by definition of $\mathcal{F}$, we have
\begin{equation}\label{eq4.1}
D_n:=\eta(1,A_n)=\left\{s(u)\star u:u\in A_n\right\}\in\mathcal{F}.
\end{equation}
In particular, $D_n\subset\mathcal{P}_m$ for every $n\in\mathbb{N}^+$. Since $J(s\star u)=J(u)$ for all $s\in\mathbb{R}$ and all $u\in S_m$, it follows that
\begin{equation*}
\max\limits_{D_n}J=\max\limits_{A_n}J\to E_{m,\mathcal{F}}\text{ as }n\to\infty
\end{equation*}
and thus $\{D_n\}\subset\mathcal{F}$ is another minimizing sequence of $E_{m,\mathcal{F}}$. Hence by the minimax principle (see \cite{Ghoussoub1993}), we can obtain a Palais-Smale sequence $\{v_n\}\subset S_m$ for $J$ at the level $E_{m,\mathcal{G}}$ such that dist$_X(v_n,D_n)\to0$ as $n\to\infty$. Define
\begin{equation*}
s_n:=s(v_n)\text{ and }u_n:=s_n\star v_n=s(v_n)\star v_n.
\end{equation*}

We claim that there exists $C>0$ such that $e^{-s_n}\leq C$ for any $n\in\mathbb{N}^+$. Firstly, we notice that
\begin{equation*}
e^{-s_n}={\left(\frac{\int_{\mathbb{R}^N}{\vert\nabla u_n\vert}^pdx}{\int_{\mathbb{R}^N}{\vert\nabla v_n\vert}^pdx}\right)}^{\frac{1}{p_c+p}}.
\end{equation*}
It follows from $\{u_n\}\subset\mathcal{P}_m$ and Lemma \ref{lem2.5}(ii) that $\left\{\Vert\nabla u_n\Vert_{L^p(\mathbb{R}^N)}\right\}$ has a positive bound from below. Concerning the term of $\{v_n\}$, since $D_n\subset\mathcal{P}_m$ for any $n\in\mathbb{N}^+$, we have
\begin{equation*}
\max\limits_{D_n}I=\max\limits_{D_n}J\to E_{m,\mathcal{F}}\text{ as }n\to\infty.
\end{equation*}
 It follows from Lemma \ref{lem2.5}(iv) that $\{D_n\}$ is uniformly bounded in $\mathscr{X}$. Since dist$_\mathscr{X}(v_n,D_n)\to0$, we have $\sup_{n\in\mathbb{N}^+}{\Vert\nabla v_n\Vert}_{L^p(\mathbb{R}^N)}<+\infty$. Therefore, the claim is proved.

Since $\{u_n\}\subset\mathcal{P}_m$, we have $I(u_n)=J(u_n)=J(v_n)\to E_{m,\mathcal{F}}$ as $n\to\infty$. Thus we only need to prove that $\{u_n\}$ is a Palais-Smale sequence for $I$ on $S_m$. For any $\psi\in T_{u_n}S_m$, we have
\begin{equation*}
\int_{\mathbb{R}^N}v_n[(-s_n)\star\psi]dx=\int_{\mathbb{R}^N}(s_n\star v_n)\psi dx=\int_{\mathbb{R}^N}u_n\psi dx=0\text{ as }n\to\infty,
\end{equation*}
which means that $(-s_n)\star\psi\in T_{v_n}S_m$. Furthermore, it follows from above claim that
\begin{equation*}
{\Vert(-s_n)\star\psi\Vert}_\mathscr{X}\leq\max\left\{C^{\frac{p_c+p}{p}}, C^{\frac{p_c}{p}}, 1\right\}{\Vert\psi\Vert}_\mathscr{X},
\end{equation*}
where ${\Vert\cdot\Vert}_{u,*}$ is the dual norm of ${\left(T_uS_m\right)}^*$. From Lemma \ref{lem4.3}, we further have
\begin{equation*}
     \begin{aligned}
          {\Vert dI(u_n)\Vert}_{u_n,*} &=\sup\limits_{\psi\in T_{u_n}S_m, {\Vert\psi\Vert}_\mathscr{X}\leq1}\left\vert dI(u_n)[\psi]\right\vert=\sup\limits_{\psi\in T_{u_n}S_m, {\Vert\psi\Vert}_\mathscr{X}\leq1}\left\vert dI(s_n\star v_n)[s_n\star((-s_n)\star\psi)]\right\vert\\
          & =\sup\limits_{\psi\in T_{u_n}S_m, {\Vert\psi\Vert}_\mathscr{X}\leq1}\left\vert dJ(v_n)[(-s_n)\star\psi]\right\vert\leq{\Vert dJ(v_n)\Vert}_{v_n,*}\sup\limits_{\psi\in T_{u_n}S_m, {\Vert\psi\Vert}_\mathscr{X}\leq1}\left\vert(-s_n)\star\psi\right\vert\\
          & \leq\max\left\{C^{\frac{p_c+p}{p}}, C^{\frac{p_c}{p}}, 1\right\}{\Vert dJ(v_n)\Vert}_{v_n,*}.
     \end{aligned}
\end{equation*}
Since $\{v_n\}\subset S_m$ is a Palais-Smale sequence of $J$, it clearly holds ${\Vert dI(u_n)\Vert}_{u_n,*}\to0$ as $n\to\infty$.

Finally, we notice that the class of all set-with-one-element included in $S_m$ with $B=\emptyset$ is a homotopy-stable family of $S_m$. Moreover, if $f$ is odd, taking this especially choice for $\mathcal{F}$ and by Lemma \ref{lem2.4}(iv), $J(u)$ is an even function in $u\in S_m$. Thus, we can choose a minimizing sequence $\{A_n\}\subset\mathcal{F}$ that consists of nonnegative functions and the sequence $\{D_n\}$ defined in \eqref{eq4.1} retains this property. Since dist$_\mathscr{X}(v_n,D_n)\to0$ as $n\to\infty$, there exists a Palais-Smale sequence $\{u_n\}\subset\mathcal{P}_m$ for $I|_{S_m}$ at the level $E_{m,\mathcal{F}}$ satisfying
\begin{equation*}
{\Vert u^-_n\Vert}^2_{L^2(\mathbb{R}^N)}={\Vert s(v_n)\star v^-_n\Vert}^2_{L^2(\mathbb{R}^N)}={\Vert v^-_n\Vert}^2_{L^2(\mathbb{R}^N)}\to0\text{ as }n\to\infty.
\end{equation*}
This completes the proof.
\end{proof}

\noindent\textit{Proof of Lemma 4.1.}
Lemma \ref{lem4.5} plays an important role in the especial case where $\mathcal{F}$ is the class of all set-with-one-element included in $S_m$. Since $E_m>0$, we only need to prove $E_{m,\mathcal{F}}=E_m$. Firstly, we have
\begin{equation*}
E_{m,\mathcal{F}}=\inf\limits_{A\in\mathcal{F}}\max\limits_{u\in A}J(u)=\inf\limits_{u\in S_m}I(s(u)\star u).
\end{equation*}
Since $s(u)\star u\in\mathcal{P}_m$ for each $u\in S_m$, we have $I(s(u)\star u)\geq E_m$, which implies that $E_{m,\mathcal{F}}\geq E_m$. Moreover, $s(u)=0$ and $I(u)=I(0\star u)\geq E_{m,\mathcal{F}}$ for all $u\in\mathcal{P}_m$, which leads to $E_m\geq E_{m,\mathcal{F}}$. This completes the proof.
$\hfill\square$
\\ \hspace*{\fill}

Since $\mathscr{X}=W^{1,p}(\mathbb{R}^N)\cap L^2(\mathbb{R}^N)$, we have $\mathscr{X}\hookrightarrow L^2(\mathbb{R}^N)$. Denote the $L^2(\mathbb{R}^N)$ norm and scalar product by $\Vert\cdot\Vert_{L^2(\mathbb{R}^N)}$ and $(\cdot,\cdot)$, respectively. For give point $u\in\mathscr{X}$, set
\begin{equation*}
T_uS_m=\left\{v\in\mathscr{X}: (u,v)=\int_{\mathbb{R}^N}uvdx=0\right\}.
\end{equation*}
Denote by $\pi_u$ the restriction to $\mathscr{X}$ of the orthogonal projection onto $T_uS_m$, that is,
$$\pi_uv:=v-(v,u)u =v-u\int_{\mathbb{R}^N}uvdx\text{ for all }v\in\mathscr{X}.$$
Denote by $I|_{S_m}$ the trace of $I$ on $S_m$. Since the functional $I: \mathscr{X}\rightarrow\mathbb{R}$ is class of $C^1$ on $\mathscr{X}$, then $I|_{S_m}$ is a $C^1$ functional on $S_m$. Moreover, for any $u\in{S_m}$ and $w\in T_uS_m$, we have
\begin{equation*}
\langle dI|_{S_m}(u), w\rangle=\langle dI(u), w\rangle.
\end{equation*}
Hence we have following Lions lemma (see \cite{LionsARMAII1983}).

\begin{lemma}\label{Lemma 3}
Let $\{u_n\}$ be a sequence in $S_m$ which is bounded in $\mathscr{X}$. Then following statement are equivalent:

(i) ${\Vert dI|_{S_m}(u_n)\Vert}_{u_n,*}\rightarrow0$ as $n\rightarrow+\infty$.

(ii) $dI(u_n)-\langle dI(u_n), u_n\rangle u_n\rightarrow0$ in $\mathscr{X}^*$.
\end{lemma}

\begin{proof}
For any $v\in\mathscr{X}$, it has the unique decomposition $v=(v,u)u+\pi_uv$ with $u\in S_m$ and $\pi_u\in T_uS_m$. Since $\vert(u,v)\vert\leq m\Vert v\Vert_{L^2(\mathbb{R}^N)}\leq m\Vert v\Vert_{\mathscr{X}}$, then for all $v\in\mathscr{X}, u\in S_m$, we have
\begin{equation*}
\Vert \pi_uv\Vert_{\mathscr{X}}=\Vert v-(v,u)u\Vert_{\mathscr{X}}\leq\Vert v\Vert_{\mathscr{X}}+\vert(v,u)\vert\Vert u\Vert_{\mathscr{X}}\leq(1+m\Vert u\Vert_{\mathscr{X}})\Vert v\Vert_{\mathscr{X}}.
\end{equation*}
Let $d\widetilde{I}(u)=dI(u)-\langle dI(u), u\rangle u$, then for each $v\in\mathscr{X}$, we have
\begin{equation*}
     \begin{aligned}
     d\widetilde{I}(u)v &=\left(dI(u)-\langle dI(u), u\rangle u\right)v\\
     & =\int_{\mathbb{R}^N}\vert\nabla u\vert^{p-2}\nabla u\cdot\nabla vdx+\int_{\mathbb{R}^N}\vert u\vert^{p-2}uvdx-\int_{\mathbb{R}^N}F(u)vdx-\langle dI(u), u\rangle\int_{\mathbb{R}^N}uvdx.
     \end{aligned}
\end{equation*}
Clearly, $d\widetilde{I}(u)\in\mathscr{X}^*$ and $\langle d\widetilde{I}(u), w\rangle=\langle dI|_{S_m}(u), w\rangle$ for all $w\in T_uS_m$, which implies that ${\Vert dI|_{S_m}(u)\Vert}_{u,*}\leq{\Vert d\widetilde{I}(u)\Vert}_{\mathscr{X}^*}$ for all $u\in S_m$ and thus we deduce (ii) $\Rightarrow$ (i).

Suppose $\{u_n\}\subset S_m$ is a bounded sequence such that ${\Vert dI|_{S_m}(u_n)\Vert}_{u_n,*}\rightarrow0$ and $\langle d\widetilde{I}(u_n), v\rangle=\langle dI(u_n), \pi_{u_n}v\rangle$ for any $v\in\mathscr{X}$. Thus, we have
\begin{equation*}
     \begin{aligned}
     \vert\langle d\widetilde{I}(u_n), v\rangle\vert &=\vert\langle dI(u_n), v\rangle-\langle \langle dI(u_n), u_n\rangle u_n, v\rangle\vert\\
     & =\vert\langle dI(u_n), v\rangle-\langle dI(u_n), u_n\rangle \langle u_n, v\rangle\vert\\
     & =\vert\langle dI(u_n), v\rangle-\langle dI(u_n), u_n\rangle \langle u_n, (v,u_n)u_n\rangle\vert\\
     & =\vert\langle dI(u_n), v-m(v,u_n)u_n\rangle\vert\\
     & \leq \Vert dI(u_n)\Vert_{\mathscr{X}^*}\Vert v-m(v,u_n)u_n\Vert_{\mathscr{X}}\\
     & \leq \Vert dI|_{S_m}(u_n)\Vert_{u_n,*}\left(\Vert v\Vert_{\mathscr{X}}+m^2\Vert v\Vert_{\mathscr{X}}\Vert u_n\Vert_{\mathscr{X}}\right)\\
     & =\Vert dI|_{S_m}(u_n)\Vert_{u_n,*}\left(1+m^2\Vert u_n\Vert_{\mathscr{X}}\right)\Vert v\Vert_{\mathscr{X}}.
     \end{aligned}
\end{equation*}
Since $\{u_n\}$ is bounded, then $\lim_{n\to\infty}\Vert d\widetilde{I}(u_n)\Vert_{\mathscr{X}^*}=0$ and thus we deduce (i) $\Rightarrow$ (ii).
\end{proof}

\begin{lemma}\label{lem2.6}
Assume $N\geq3$ and $f$ satisfies $(f0)-(f2)$. If $\{u_n\}\subset\mathscr{X}$ is bounded and $u_n\to u$ a.e. in $\mathbb{R}^N$ as $n\to\infty$ for some $u\in\mathscr{X}$, then
\begin{equation}\label{eq2.14}
\lim\limits_{n\to+\infty}\int_{\mathbb{R}^N}\vert F(u_n)-F(u_n-u)-F(u)\vert dx=0.
\end{equation}
\end{lemma}

\begin{proof}
Since $\{u_n\}$ is bounded and $u_n\to u$ a.e. in $\mathbb{R}^N$ as $n\to\infty$ for some $u\in\mathscr{X}$, there exists $M>0$ large enough such that $$\sup_{n\geq1}{\Vert u_n\Vert}_\mathscr{X}<M,~\sup_{n\geq1}{\Vert u_n-u\Vert}_\mathscr{X}<M\text{ and }\sup_{n\geq1}{\Vert u\Vert}_\mathscr{X}<M.$$
Thanks to $(f0)-(f2)$, there exists a positive constant $C$ such that
\begin{equation}\label{eq(*)}
\vert f(t)\vert\leq C\left(\vert t\vert+{\vert t\vert}^{p^*-1}\right)\text{ for all } t\in\mathbb{R}.
\end{equation}
Fixed any $a, b\in\mathbb{R}$, by \eqref{eq(*)} and Young's inequality, for any $\varepsilon>0$, we have
\begin{equation*}
     \begin{aligned}
          \vert F(a+b)-F(a)\vert &=\left\vert\int^1_0f(a+\theta b)bd\theta\right\vert\\
          & \leq C\int^1_0\left(\vert a+\theta b\vert+{\vert a+\theta b\vert}^{p^*-1}\right)\vert b\vert d\theta\\
          & \leq C\int^1_0\left[\vert a\vert+\vert\theta\vert\vert b\vert+2^{p^*}\left({\vert a\vert}^{p^*-1}+{\vert\theta b\vert}^{p^*-1}\right)\right]\vert b\vert d\theta\\
          & \leq C \left[\vert a\vert+\vert b\vert+2^{p^*}\left({\vert a\vert}^{p^*-1}+{\vert b\vert}^{p^*-1}\right)\right]\vert b\vert\\
          & \leq\varepsilon C\left(a^2+{\vert2a\vert}^{p^*}\right)+C\left[\left(1+\frac{1}{\varepsilon}\right)b^2+\left(1+{\varepsilon}^{1-p^*}\right){\vert2b\vert}^{p^*}\right]\\
          & =:\varepsilon\varphi(a)+\psi_\varepsilon(b)
     \end{aligned}
\end{equation*}
In particular, $\vert F(b)\vert\leq\psi_\varepsilon(b)$ for all $b\in\mathbb{R}$. By the Gagliardo-Nirenberg inequality, we know that $\int_{\mathbb{R}^N}\varphi(u_n-u)dx$ is bounded uniformly in $\varepsilon$ and $n$, $\int_{\mathbb{R}^N}\psi_\varepsilon(u)dx<+\infty$ for any $\varepsilon>0$ and $F(u)\in L^1(\mathbb{R}^N)$. Using the Brezis-Lieb Theorem (see \cite{BrezisPAMS1983}),  \eqref{eq2.14} clearly holds.
\end{proof}

\begin{lemma}\label{lemma_xin}
Assume $N\geq3$, $f$ satisfies $(f0)-(f5)$. Suppose ${\Vert u_0\Vert}^2_{L^2(\mathbb{R}^N)}=m$ and
\begin{equation}\label{(p_m)}
-\Delta_{p}u_0+{\vert u_0\vert}^{p-2}u_0+\mu_0u_0=f(u_0)
\end{equation}
If there exists a sufficiently small constant $m_0>0$ such that $m\in(0,m_0)$, then $\mu_0>0$.
\end{lemma}

\begin{proof}
By $(f0)-(f2)$, for any positive constants $\varepsilon$ and $\delta$, there exists $c_1({\varepsilon}), c_2({\varepsilon})>0$ such that
\begin{equation*}
\vert F(t)\vert\leq c_1{\vert t\vert}^{p_1}+c_2{\vert t\vert}^{p_2},
\end{equation*}
where $p_1=p_*+\varepsilon, p_2=p^*-\delta$. We can choose test function $u_0$ given in \eqref{(p_m)} to obtain
\begin{equation}\label{eq55}
\int_{\mathbb{R}^N}\vert\nabla u_0\vert^pdx+\int_{\mathbb{R}^N}\vert u_0\vert^pdx+\mu_0\int_{\mathbb{R}^N}u_0^2dx=\int_{\mathbb{R}^N}f(u_0)u_0dx.
\end{equation}
It follows from the Pohozaev identity \cite{GueddaNA1989} that
\begin{equation}\label{eq56}
\frac{N-p}{p}\int_{\mathbb{R}^N}\vert\nabla u_0\vert^pdx+\frac{N}{p}\int_{\mathbb{R}^N}\vert u_0\vert^pdx+\frac{N}{2}\mu_0\int_{\mathbb{R}^N}u_0^2dx=N\int_{\mathbb{R}^N}F(u_0)dx.
\end{equation}
Combining \eqref{eq55}, \eqref{eq56} with $(f5)$, we can obtain that
\begin{equation}\label{eq57}
     \begin{aligned}
     \frac{p_c+p}{p}&\int_{\mathbb{R}^N}\vert\nabla u_0\vert^pdx+\frac{p_c}{p}\int_{\mathbb{R}^N}\vert u_0\vert^pdx
      \leq\frac{N}{2}\int_{\mathbb{R}^N}\vert f(u_0)u_0\vert dx+N\int_{\mathbb{R}^N}\vert F(u_0)\vert dx\\
     & \leq\frac{(\gamma+2)N}{2}\int_{\mathbb{R}^N}\vert F(u_0)\vert dx\leq C_1\int_{\mathbb{R}^N}\vert u_0\vert^{p_1}dx+C_2\int_{\mathbb{R}^N}\vert u_0\vert^{p_2}dx,
     \end{aligned}
\end{equation}
and
\begin{equation}\label{eq58}
     \begin{aligned}
     &\left(\frac{N-p}{p}-\frac{N}{\gamma}\right)\int_{\mathbb{R}^N}\vert\nabla u_0\vert^pdx+\left(\frac{1}{p}-\frac{1}{\gamma}\right)N\int_{\mathbb{R}^N}\vert u_0\vert^pdx\\
     & =-\left(\frac{1}{2}-\frac{1}{\gamma}\right)N\mu_0\int_{\mathbb{R}^N}\vert u_0\vert^2dx+\frac{N}{\gamma}\int_{\mathbb{R}^N}\gamma F(t)-f(t)tdx
     \end{aligned}
\end{equation}
Using sharp Gagliardo-Nirenberg inequalities \cite{AguehNDE2008}, we further have
\begin{equation}\label{eq59}
     \begin{aligned}
     & C_1\int_{\mathbb{R}^N}\vert u_0\vert^{p_1}dx\leq C_3\Vert\nabla u_0\Vert_{L^p}^{\frac{Np(p_1-2)}{Np+2p-2N}}\Vert u_0\Vert_{L^2}^{\frac{2Np-2p_1(N-p)}{Np+2p-2N}},\\
     & C_2\int_{\mathbb{R}^N}\vert u_0\vert^{p_2}dx\leq C_4\Vert\nabla u_0\Vert_{L^p}^{\frac{Np(p_2-2)}{Np+2p-2N}}\Vert u_0\Vert_{L^2}^{\frac{2Np-2p_2(N-p)}{Np+2p-2N}}.
     \end{aligned}
\end{equation}
For the convenience of calculation, set
\begin{equation*}
g(x)=\frac{Np(x-2)}{Np+2p-2N} \  \text{ and } \  h(x)=\frac{2Np-2x(N-p)}{Np+2p-2N}.
\end{equation*}
Then $g(p_1), g(p_2), h(p_1), h(p_2)>0$ by a direct computation. It follows from \eqref{eq57} and \eqref{eq59} that
\begin{equation*}
     \begin{aligned}
     &\frac{p_c+p}{p}\int_{\mathbb{R}^N}\vert\nabla u_0\vert^pdx-C_3\Vert\nabla u_0\Vert_{L^p}^{g(p_1)}\Vert u_0\Vert_{L^2}^{h(p_1)}-C_4\Vert\nabla u_0\Vert_{L^p}^{g(p_2)}\Vert u_0\Vert_{L^2}^{h(p_2)}\\
     & \leq\frac{p_c+p}{p}\int_{\mathbb{R}^N}\vert\nabla u_0\vert^pdx-C_1\int_{\mathbb{R}^N}\vert u_0\vert^{p_1}dx-C_2\int_{\mathbb{R}^N}\vert u_0\vert^{p_2}dx\leq0,
     \end{aligned}
\end{equation*}
which leads to
\begin{equation}\label{eq61}
\frac{p_c+p}{p}\int_{\mathbb{R}^N}\vert\nabla u_0\vert^pdx-C_3\Vert\nabla u_0\Vert_{L^p}^{g(p_1)}\Vert u_0\Vert_{L^2}^{h(p_1)}-C_4\Vert\nabla u_0\Vert_{L^p}^{g(p_2)}\Vert u_0\Vert_{L^2}^{h(p_2)}\leq0.
\end{equation}
Thus, there exists $m_0>0$ small enough such that $\Vert u_0\Vert_{L^2}^2=m\le m_0$, then $\Vert\nabla u_0\Vert_{L^p}$ must be large. Moreover, using the interpolation inequalities for $L^p$ and the sharp Gagliardo-Nirenberg inequalities for $L^{p_1}$, we obtain
\begin{equation}\label{eq62}
     \begin{aligned}
     \int_{\mathbb{R}^N}\vert u_0\vert^pdx &\leq\Vert u_0\Vert_{L^2}^{p\theta}\Vert u_0\Vert_{L^{p_1}}^{p(1-\theta)}\\
     & \leq C_0\Vert u_0\Vert_{L^2}^{p\theta}\Vert u_0\Vert_{L^2}^{p(1-\theta)\frac{h(p_1)}{p_1}}\Vert\nabla u_0\Vert_{L^p}^{p(1-\theta)\frac{g(p_1)}{p_1}}\text{ for }\theta\in(0,1).
     \end{aligned}
\end{equation}
It follows from \eqref{eq58}, \eqref{eq61} and \eqref{eq62} that
\begin{equation*}
     \begin{aligned}
     -\left(\frac{1}{2}-\frac{1}{\gamma}\right)N\mu_0\int_{\mathbb{R}^N}\vert u_0\vert^2dx &\leq\left(\frac{N-p}{p}-\frac{N}{\gamma}\right)\Vert\nabla u_0\Vert_{L^p}^p+\left(\frac{1}{p}-\frac{1}{\gamma}\right)N\Vert u_0\Vert_{L^p}^p\\
     & \leq\left(\frac{N-p}{p}-\frac{N}{\gamma}\right)\Vert\nabla u_0\Vert_{L^p}^p+C_5\Vert u_0\Vert_{L^2}^{p\left[\theta+(1-\theta)\frac{h(p_1)}{p_1}\right]}\Vert\nabla u_0\Vert_{L^p}^{p(1-\theta)\frac{g(p_1)}{p_1}}.
     \end{aligned}
\end{equation*}
A direct computation yields that
\begin{equation*}
\frac{g(p_1)}{p_1}<1, \frac{N-p}{p}-\frac{N}{\gamma}<0 \text{ and } \ \frac{Np-p_1(N-p)}{Np+2p-2N}>0.
\end{equation*}
We can choose $m_0>0$ sufficiently small such that $\Vert u_0\Vert_{L^2}^2=m<m_0$, then $\mu_0=0$ clearly contradicts with above inequality, which leads to $\mu_0>0$.
\end{proof}

\begin{lemma}\label{lem4.6}
Take any $m\in(0,m_0)$ for some $m_0>0$ small and let $\{u_n\}\subset S_m$ be any bounded Palais-Smale sequence for the constrained functional $I|_{S_m}$ at the level $E_m>0$ satisfying $P(u_n)\to0$. If condition $(f5)$ holds, then there exists $u\in S_m$ and $\mu>0$ such that, up to a subsequence and some translations in $\mathbb{R}^N$, $u_n\to u$ strongly in $\mathscr{X}$ and $-\Delta_{p}u+{\vert u\vert}^{p-2}u=f(u)-\mu u$.
\end{lemma}

\begin{proof}
Since $\{u_n\}\subset S_m$ is bounded in $\mathscr{X}$, without loss of generality,  following limits exist
\begin{equation*}
\lim\limits_{n\to+\infty}{\Vert\nabla u_n\Vert}_{L^p(\mathbb{R}^N)}, \ \lim\limits_{n\to+\infty}{\Vert u_n\Vert}_{L^p(\mathbb{R}^N)}, \ \lim\limits_{n\to+\infty}\int_{\mathbb{R}^N}F(u_n)dx, \ \lim\limits_{n\to+\infty}\int_{\mathbb{R}^N}f(u_n)u_ndx.
\end{equation*}
It follows from ${\Vert dI(u_n)\Vert}_{u_n,*}\to0$ and $\mathscr{X}\hookrightarrow L^2(\mathbb{R}^N)$ that $\mathscr{X}$ is a Hilbert space in the sense of $L^2(\mathbb{R}^N)$. By Lemma \ref{Lemma 3}, we have $Q(u_n):=dI(u_n)-\langle dI(u_n),u_n \rangle u_n\to0$ in $\mathscr{X}^*$ as $n\to\infty$. Then for any $v\in\mathscr{X}$, it holds
\begin{equation*}
     \int_{\mathbb{R}^N}\vert\nabla u_n\vert^{p-2}\nabla u_n\cdot\nabla vdx+\int_{\mathbb{R}^N}\vert u_n\vert^{p-2}u_nvdx+\mu_n\int_{\mathbb{R}^N}u_nvdx-\int_{\mathbb{R}^N}F(u_n)vdx\to0,
\end{equation*}
For convenience, we rewrite above limit as follows
\begin{equation*}
-\Delta_{p}u_n+{\vert u_n\vert}^{p-2}u_n+\mu_nu_n-f(u_n)\to0\text{ in } \mathscr{X}^*,
\end{equation*}
where
\begin{equation*}
     \mu_n:=-\langle dI(u_n), u_n \rangle=\frac{1}{m}\left(\int_{\mathbb{R}^N}f(u_n)u_ndx-\int_{\mathbb{R}^N}\vert\nabla u_n\vert^pdx-\int_{\mathbb{R}^N}\vert u_n\vert^pdx\right).
\end{equation*}
Since $\{u_n\}$ is bounded in $\mathscr{X}$, then $\mu_n\to\mu$ as $n\to\infty$ for some $\mu\in\mathbb{R}$. Moreover, we have
\begin{equation}\label{eq4.2}
-\Delta_{p}u_n(\cdot+y_n)+{\vert u_n(\cdot+y_n)\vert}^{p-2}u_n(\cdot+y_n)+\mu u_n(\cdot+y_n)-f(u_n(\cdot+y_n))\to0\text{ in }\mathscr{X}^*
\end{equation}
for any $\{y_n\}\subset\mathbb{R}^N$.

Firstly, we claim that $\{u_n\}$ is non-vanishing. Indeed, if $\{u_n\}$ is vanishing, then Lemma \ref{Lemma I.1} implies that $u_n\to0$ in ${L^{p_*}(\mathbb{R}^N)}$. Using Lemma \ref{lem2.1}(ii) and the fact that $P(u_n)\to0$ as $n\to\infty$, we have $\int_{\mathbb{R}^N}F(u_n)dx\to0$ as $n\to\infty$ and
\begin{equation*}
\frac{p_c+p}{p}\int_{\mathbb{R}^N}{\vert\nabla u_n\vert}^pdx+\frac{p_c}{p}\int_{\mathbb{R}^N}{\vert u_n\vert}^pdx=P(u_n)+\frac{N}{2}\int_{\mathbb{R}^N}\widetilde{F}(u_n)dx\to0\text{ as }n\to\infty,
\end{equation*}
which implies that $\int_{\mathbb{R}^N}{\vert\nabla u_n\vert}^pdx\to0$ and $\int_{\mathbb{R}^N}{\vert u_n\vert}^pdx\to0$ as $n\to\infty$. Hence, we have
\begin{equation*}
E_m=\lim\limits_{n\to+\infty}I(u_n)=\frac{1}{p}\lim\limits_{n\to+\infty}\int_{\mathbb{R}^N}{\vert\nabla u_n\vert}^pdx+\frac{1}{p}\lim\limits_{n\to+\infty}\int_{\mathbb{R}^N}{\vert u_n\vert}^pdx-\lim\limits_{n\to+\infty}\int_{\mathbb{R}^N}F(u_n)dx=0,
\end{equation*}
which contradicts with $E_m>0$ and thus the claim is proved.

Since $\{u_n\}$ is non-vanishing, there exists $\{y^1_n\}\subset\mathbb{R}^N$ and $w^1\in B_m\backslash\{0\}$ such that $u_n(\cdot+y^1_n)\rightharpoonup w^1$ in $\mathscr{X}$, $u_n(\cdot+y^1_n)\to w^1$ in $L^q_\text{loc}(\mathbb{R}^N)$ for any $q\in[1,p^*)$ and $u_n(\cdot+y^1_n)\to w^1$ a.e. in $\mathbb{R}^N$. It follows from \eqref{eq(*)} in Lemma \ref{lem2.6} and  Theorem A.I. in \cite{LionsARMAI1983} that
\begin{equation}\label{eq(**)}
\lim\limits_{n+\infty}\int_{\mathbb{R}^N}\left\vert [f(u_n(\cdot+y^1_n))-f(w^1)]\varphi\right\vert dx\leq{\Vert\varphi\Vert}_{L^\infty(\mathbb{R}^N)}\lim\limits_{n+\infty}\int_{\text{spt}(\varphi)}\left\vert f(u_n(\cdot+y^1_n))-f(w^1)\right\vert dx=0
\end{equation}
for any $\varphi\in C^\infty_0(\mathbb{R}^N)$. From \eqref{eq4.2}, we further have
\begin{equation}\label{eq4.3}
-\Delta_{p}w^1+{\vert w^1\vert}^{p-2}w^1+\mu w^1=f(w^1).
\end{equation}
According to the Nehari identity and the Pohozaev identity corresponding to \eqref{eq4.3}, we have $P(w^1)=0$. Define $v^1_n:=u_n-w^1(\cdot-y^1_n)$ for each $n\in\mathbb{N}^+$, then $v^1_n(\cdot+y^1_n)\rightharpoonup0$ in $\mathscr{X}$ and
\begin{equation}\label{eq4.4}
m=\lim\limits_{n\to+\infty}{\Vert v^1_n(\cdot+y^1_n)+w^1\Vert}^2_{L^2(\mathbb{R}^N)}={\Vert w^1\Vert}^2_{L^2(\mathbb{R}^N)}+\lim\limits_{n\to+\infty}{\Vert v^1_n\Vert}^2_{L^2(\mathbb{R}^N)}.
\end{equation}

\textbf{Claim.} Still using the above symbols, we have the following equation
\begin{equation*}
\lim\limits_{n\to+\infty}{\Vert v^1_n(\cdot+y^1_n)+w^1\Vert}^p_{L^p(\mathbb{R}^N)}={\Vert w^1\Vert}^p_{L^p(\mathbb{R}^N)}+\lim\limits_{n\to+\infty}{\Vert v^1_n\Vert}^p_{L^p(\mathbb{R}^N)},
\end{equation*}
\begin{equation*}
\lim\limits_{n\to+\infty}{\Vert\nabla\left(v^1_n(\cdot+y^1_n)+w^1\right)\Vert}^p_{L^p(\mathbb{R}^N)}={\Vert\nabla w^1\Vert}^p_{L^p(\mathbb{R}^N)}+\lim\limits_{n\to+\infty}{\Vert\nabla v^1_n\Vert}^p_{L^p(\mathbb{R}^N)}.
\end{equation*}
We only need to prove that $\nabla u_n(\cdot+y^1_n)\to\nabla w^1 $ a.e. in $\mathbb{R}^N$.
Inspired by Lemma 2.1 in \cite{AlvesJMP2014}, there exists a function $\psi\in C^\infty_0(\mathbb{R}^N)$ such that $0\leq\psi(x)\leq1$ for every $x\in\mathbb{R}^N$ and
\begin{equation*}
     \psi(x)=\left\{
     \begin{aligned}
          &1, \quad \text{ if } \vert x\vert\leq1,\\
          &0, \quad \text{ if } \vert x\vert\geq2,
     \end{aligned}
     \right.
\end{equation*}
Define $\psi_R(x):=\psi(\frac{x}{R})$ for all $x\in\mathbb{R}^N$ and $R>0$. After some calculations, we have
\begin{equation*}
     \begin{aligned}
          T_n &=\langle Q(u_n), (u_n-w^1)\psi_R \rangle+\int_{\mathbb{R}^N}{\vert\nabla u_n\vert}^{p-2}\nabla u_n\cdot\nabla\psi_R(u_n-w^1)dx\\
          & +\int_{\mathbb{R}^N}{\vert u_n\vert}^{p-2}u_n(u_n-w^1)\psi_Rdx-\int_{\mathbb{R}^N}{\vert\nabla w^1\vert}^{p-2}\nabla w^1\cdot\nabla(u_n-w^1)\psi_Rdx\\
          & +\int_{\mathbb{R}^N}f(u_n)(u_n-w^1)\psi_Rdx+\mu_n\int_{\mathbb{R}^N}u_n(u_n-w^1)\psi_Rdx,
     \end{aligned}
\end{equation*}
where
\begin{equation*}
     \begin{aligned}
          T_n &:=\int_{\mathbb{R}^N}\langle {\vert\nabla u_n\vert}^{p-2}\nabla u_n-{\vert\nabla w^1\vert}^{p-2}\nabla w^1, \nabla u_n-\nabla w^1 \rangle\psi_Rdx\\
          & +\int_{\mathbb{R}^N}\left({\vert u_n\vert}^{p-2}u_n-{\vert w^1\vert}^{p-2}w^1\right)\left(u_n-w^1\right)\psi_Rdx.
     \end{aligned}
\end{equation*}
Thanks to \eqref{eq(**)} and the fact that $\lim_{n\to\infty}Q(u_n)=0$ in $\mathscr{X}^*$, we have
\begin{equation*}
\int_{\mathbb{R}^N}f(u_n)(u_n-w^1)\psi_Rdx=o_n(1)\text{ and }\langle Q(u_n),(u_n-w^1)\psi_R \rangle=o_n(1).
\end{equation*}
By Sobolev compact embeddings $\mathscr{X}\hookrightarrow L^q_{\text{loc}}(\mathbb{R}^N)$ for any $q\in[1,p^*)$ and weak convergence $u_n(\cdot+y^1_n)\rightharpoonup w^1$ in $\mathscr{X}$, we have
\begin{equation*}
     \begin{aligned}
          \int_{\mathbb{R}^N}{\vert\nabla u_n\vert}^{p-2}\nabla u_n\cdot\nabla\psi_R(u_n-w^1)dx=o_n(1), \ \ \int_{\mathbb{R}^N}{\vert u_n\vert}^{p-2}u_n(u_n-w^1)\psi_Rdx=o_n(1),\\
          \mu_n\int_{\mathbb{R}^N}u_n(u_n-w^1)\psi_Rdx=o_n(1), \ \ \int_{\mathbb{R}^N}{\vert\nabla w^1\vert}^{p-2}\nabla w^1\cdot\nabla(u_n-w^1)\psi_Rdx=o_n(1).
     \end{aligned}
\end{equation*}
Then we have $T_n=o_n(1)$. According to \cite{Damascelli1998}, following inequality holds
\begin{equation*}
     \langle {\vert\eta\vert}^{p-2}\eta-{\vert\xi\vert}^{p-2}\xi, \eta-\xi \rangle\geq\left\{
     \begin{aligned}
          &C_1\vert{\eta-\xi}\vert^p, \qquad\qquad\qquad\  \text{ if } p\geq2,\\
          &C_2{(\vert\eta\vert+\vert\xi\vert)}^{p-2}{\vert\eta-\xi\vert}^2, \quad \text{ if } 1<p<2,
     \end{aligned}
     \right.
\end{equation*}
where constants $C1$, $C2>0$. Hence, up to a subsequence, we have $
\nabla u_n(\cdot+y^1_n)\to\nabla w^1$ a.e. in $B(0,R)$. By the arbitrariness of $R$ and up to a subsequence, we have
\begin{equation*}
\nabla u_n(\cdot+y^1_n)\to\nabla w^1 \ \text{ a.e. in } \mathbb{R}^N.
\end{equation*}
Since $\{u_n\}$ is bounded in $\mathscr{X}$, $u_n(\cdot+y^1_n)\to w^1$ a.e. in $\mathbb{R}^N$ and $\nabla u_n(\cdot+y^1_n)\to\nabla w^1$ a.e. in $\mathbb{R}^N$, by the Brezis-Lieb lemma in \cite{Willem1997}, we have completed the proof of the Claim.

Using Lemma \ref{lem2.6}, we have
\begin{equation*}
\lim\limits_{n\to+\infty}\int_{\mathbb{R}^N}F(u_n(\cdot+y^1_n))dx=\int_{\mathbb{R}^N}F(w^1)dx +\lim\limits_{n\to+\infty}\int_{\mathbb{R}^N}F(v^1_n(\cdot+y^1_n))dx.
\end{equation*}
We can combine this with the Claim to obtain that
\begin{equation}\label{eq4.5}
E_m=\lim\limits_{n\to+\infty}I(u_n)=\lim\limits_{n\to+\infty}I(u_n(\cdot+y^1_n))=I(w^1) +\lim\limits_{n\to+\infty}I(v^1_n(\cdot+y^1_n))=I(w^1)+\lim\limits_{n\to+\infty}I(v^1_n).
\end{equation}

We next show $\lim_{n\to+\infty}I(v^1_n)\geq0$. By contradiction we assume that $\lim_{n\to+\infty}I(v^1_n)<0$. Thus $\{v^1_n\}$ is non-vanishing. Up to a subsequence, there exists $\{y^2_n\}\subset\mathbb{R}^N$ such that
\begin{equation*}
\lim\limits_{n\to+\infty}\int_{B(y^2_n,1)}{\vert v^1_n\vert}^pdx>0.
\end{equation*}
Since $v^1_n(\cdot+y^1_n)\to0$ as $n\to\infty$ in $L^p_\text{loc}(\mathbb{R}^N)$, then $\vert y^2_n-y^1_n\vert\to+\infty$. Up to a subsequence, there exists some $w^2\in B_m\backslash\{0\}$ such that $v^1_n(\cdot+y^2_n)\rightharpoonup w^2$ in $\mathscr{X}$. Thanks to \eqref{eq4.2} and
\begin{equation*}
u_n(\cdot+y^2_n)=v^1_n(\cdot+y^2_n)+w^1(\cdot-y^1_n+y^2_n)\rightharpoonup w^2 \ \text{ in } \mathscr{X},
\end{equation*}
we have $P(w^2)=0$ and $I(w^2)>0$. Set $v^2_n:=v^1_n-w^2(\cdot-y^2_n)=u_n-\sum^2_{i=1}w^i(\cdot-y^i_n)$, then
\begin{equation*}
\lim\limits_{n\to+\infty}{\Vert v^2_n\Vert}^p_{L^p(\mathbb{R}^N)}=\lim\limits_{n\to+\infty}{\Vert u_n\Vert}^p_{L^p(\mathbb{R}^N)}-\sum^2_{i=1}{\Vert w^i\Vert}^p_{L^p(\mathbb{R}^N)},
\end{equation*}
\begin{equation*}
\lim\limits_{n\to+\infty}{\Vert\nabla v^2_n\Vert}^p_{L^p(\mathbb{R}^N)}=\lim\limits_{n\to+\infty}{\Vert\nabla u_n\Vert}^p_{L^p(\mathbb{R}^N)}-\sum^2_{i=1}{\Vert\nabla w^i\Vert}^p_{L^p(\mathbb{R}^N)},
\end{equation*}
which leads to
\begin{equation*}
0>\lim\limits_{n\to+\infty}I(v^1_n)=I(w^2)+\lim\limits_{n\to+\infty}I(v^2_n)>\lim\limits_{n\to+\infty}I(v^2_n).
\end{equation*}
Repeating above process, we can find an infinite sequence $\{w^k\}\subset B_m\backslash\{0\}$ such that $P(w^k)=0$ and
\begin{equation*}
\sum^{k}_{i=1}{\Vert\nabla w^i\Vert}^p_{L^p(\mathbb{R}^N)}\leq\lim\limits_{n\to+\infty}{\Vert\nabla u_n\Vert}^p_{L^p(\mathbb{R}^N)}<+\infty \ \text{ for any } k\in\mathbb{N}^+,
\end{equation*}
which is impossible. It follows from Remark \ref{re2.1} that there exists $\delta>0$ such that ${\Vert\nabla w\Vert}_{L^p(\mathbb{R}^N)}\geq\delta$ for any $w\in B_m\backslash\{0\}$ with $P(w)=0$. Thus, it holds $\lim_{n\to+\infty}I(v^1_n)\geq0$.

Set $z:={\Vert w^1\Vert}^2_{L^2(\mathbb{R}^N)}\in(0,m]$. Since $\lim_{n\to+\infty}I(v^1_n)\geq0$ and $w^1\in\mathcal{P}_z$, \eqref{eq4.5} implies that
\begin{equation*}
E_m=I(w^1)+\lim_{n\to+\infty}I(v^1_n)\geq I(w^1)\geq E_z.
\end{equation*}
Using Lemma \ref{lem3.2}, we know that $E_m$ is nonincreasing for $m>0$. Thus we have
\begin{equation}\label{eq4.6}
I(w^1)=E_z=E_m
\end{equation}
and
\begin{equation}\label{eq4.7}
\lim_{n\to+\infty}I(v^1_n)=0.
\end{equation}
It follows from \eqref{eq4.3}, \eqref{eq4.6} and Lemma \ref{lem3.4} that $\mu\geq0$. Moreover, thanks to Lemma \ref{lemma_xin}, there exists $m_0>0$ sufficiently small such that $\Vert u_0\Vert_{L^2}^2=m<m_0$, then $\mu>0$. If $z\in(0,m)$,  Lemma \ref{lem3.4} and \eqref{eq4.3} imply that $I(w^1)=E_z>E_m$, which contradicts with \eqref{eq4.6}. Hence $z={\Vert w^1\Vert}^2_{L^2(\mathbb{R}^N)}=m$ and then ${\Vert v_n^1\Vert}^2_{L^2(\mathbb{R}^N)}\to0$ as $n\to\infty$ by \eqref{eq4.4}. Note that $w^1\in\mathscr{X}$ and ${\Vert v_n^1\Vert}^2_{L^2(\mathbb{R}^N)}\to0$ as $n\to\infty$, then $v_n^1\to0$ in $L^{p_*}(\mathbb{R}^N)$ as $n\to\infty$ by Lemma \ref{Lemma I.1}. By Lemma \ref{lem2.1} (ii), we further have $\lim_{n\to+\infty}\int_{\mathbb{R}^N}F(v_n^1)dx=0$, then  ${\Vert v_n\Vert}_{W^{1,p}(\mathbb{R}^N)}\to0$ as $n\to\infty$ by \eqref{eq4.7}. Thus, $u_n(\cdot+y_n^1)\to w^1$ strongly in $\mathscr{X}$. This completes the proof.
\end{proof}

\begin{remark}\label{re4.1}
If function $f$ satisfies some stronger conditions, the proof of $\lim_{n\to+\infty}I(v_n^1)\geq0$ can be simplified. For instance, in addition to $(f0)-(f4)$, we further assume that $\widetilde{F}$ is a $C^1$ function and $\widetilde{F}'$ satisfies \eqref{eq(*)} in Lemma \ref{lem2.6}. Note that $P(w^1)=0$, similar to the proof of \eqref{eq4.5}, we have
\begin{equation*}
0=\lim\limits_{n\to+\infty}P(u_n) =P(w^1)+\lim\limits_{n\to+\infty}P(v_n^1)=\lim\limits_{n\to+\infty}P(v_n^1).
\end{equation*}
Therefore, it follows from Lemma \ref{lem2.3} that
\begin{equation*}
     \begin{aligned}
          \lim\limits_{n\to+\infty}\frac{p_c+p}{p}I(v_n^1) &=\lim\limits_{{n\to+\infty}}\left(\frac{p_c}{p}I(v_n^1)+I(v_n^1)-\frac{1}{p}P(v_n^1)\right)\\ &=\frac{N}{2p}\lim\limits_{n\to+\infty}\int_{\mathbb{R}^N}\left[f(v_n^1)v_n^1-p_*F(v_n^1)\right]dx\geq0.
     \end{aligned}
\end{equation*}
\end{remark}
With the help of Lemmas \ref{lem4.1} and \ref{lem4.6}, we now complete the proof of Theorem \ref{theo1.1}.

\noindent\textit{The proof of Theorem 1.1.}
By Lemma \ref{lem4.1} and Lemma \ref{lem2.5}(iv), for the constrained functional $I|_{S_m}$ we can obtain a bounded Palais-Smale sequence $\{u_n\}\subset\mathcal{P}_m$ at the level $E_m>0$.

(i) If the condition $(f5)$ holds, then there exists $m_0>0$ small enough such that $m\in(0,m_0)$. By Lemma \ref{lem4.6}, the proof of the existence of a ground state $u\in S_m$ at the level $E_m$ is obvious.

(ii) If $f$ is odd, then by Lemma \ref{lem4.1} implies that $\Vert u_n^-\Vert_{L^2(\mathbb{R}^N)}\to0$ as $n\to\infty$. Using Lemma \ref{lem4.6}, $\lim_{n\to\infty}\Vert u_n^-\Vert_{L^2(\mathbb{R}^N)}=0$ implies $u\geq0$, and we obtain a nonnegative ground state $u\in S_m$ at the level $E_m$. $\hfill\square$
\\ \hspace*{\fill}

Next, we give the proof of Theorem \ref{theo1.2}.

\noindent\textit{The proof of Theorem \ref{theo1.2}}.
We first show $E_m$ is strictly increasing for $m>0$. Due to Theorem \ref{theo1.1}, $E_m$ is reached by a ground state of \eqref{(P_m)} with the associated Lagrange multiplier $\mu>0$. Thus, by Lemma \ref{lem3.4}, the function $m\to E_m$ is strictly decreasing on $(0,+\infty)$. Thanks to Lemmas \ref{lem2.5}, \ref{lem3.1}, \ref{lem3.2}, \ref{lem3.5}, the rest of proof of Theorem \ref{theo1.2} can be easily proved.
 $\hfill\square$
\\ \hspace*{\fill}

\bibliographystyle{plain}

\begin{thebibliography}{00}

\bibitem{AguehNDE2008}
\textsc{M. Agueh},
\newblock{Sharp Gagliardo-Nirenberg inequalities via p-Laplacian type equation},
 {\it NoDEA Nonlinear Differential Equations and Appl.}, {\bf 15} (2008), 457-472.
\bibitem{AlvesJMP2014}
\textsc{C. O. Alves and M. Yang},
\newblock{Multiplicity and concentration of solutions for a quasilinear Choquard equation},
 {\it J. Math. Phys.}, {\bf 55} (2014), 061502.
\bibitem{Badiale2010}
\textsc{M. Badiale and E. Serra},
{\it Semilinear Elliptic Equations for Beginners: Existence Results via the Variational Approach}, Springer Science \& Business Media, (2010).
\bibitem{BartschJMPA2016}
\textsc{T. Bartsch, L. Jeanjean and N. Soave},
\newblock{Normalized solutions for a system of coupled cubic Schr\"{o}dinger equations on $\mathbb{R}^3$},
 {\it J. Math. Pures Appl.}, {\bf 106} (2016), 583-614.
\bibitem{BartschJFA2017}
\textsc{T. Bartsch and N. Soave},
\newblock{A natural constraint approach to normalized solutions of nonlinear Schr\"{o}dinger equations and systems},
 {\it J. Funct. Anal.}, {\bf 272} (2017), 4998-5037.
\bibitem{BartschJFA2018}
\textsc{T. Bartsch and N. Soave},
\newblock{Correction to ¡°A natural constraint approach to normalized solutions on nonlinear
Schr\"{o}dinger equations and systems¡±},
 {\it J. Funct. Anal.}, {\bf 275} (2018), 516-521.
\bibitem{BartschCVPDE2019}
\textsc{T. Bartsch and N. Soave},
\newblock{Multiple normalized solutions for a competing system of Schr\"{o}dinger equations},
 {\it Calc. Var. Partial Differential Equations}, {\bf 58} (2019), 1-24.
\bibitem{BellazziniPLMS2013}
\textsc{J. Bellazzini, L. Jeanjean and T. Luo},
\newblock{Existence and instability of standing waves with prescribed norm for a class of Schr\"{o}dinger-Poisson equations},
 {\it Proc. Lond. Math. Soc.}, {\bf 107} (2013), 303-339.
\bibitem{LionsARMAI1983}
\textsc{H. Berestycki and P. L. Lions},
\newblock{Nonlinear scalar field equations. I. Existence of a ground state},
 {\it Arch. Rational Mech. Anal.}, {\bf 82} (1983), 313-345.
\bibitem{LionsARMAII1983}
\textsc{H. Berestycki and P. L. Lions},
\newblock{Nonlinear scalar field equations. II. Existence of infinitely many solutions},
 {\it Arch. Rational Mech. Anal.}, {\bf 82} (1983), 347-375.
\bibitem{BrezisPAMS1983}
\textsc{H. Brezis and E. Lieb},
\newblock{A relation between pointwise convergence of functions and convergence of functionals},
 {\it Proc. Amer. Math. Soc.}, {\bf 88} (1983), 486¨C490.
\bibitem{CazenaveCMP1982}
\textsc{T. Cazenave and P. L. Lions},
\newblock{Orbital stability of standing waves for some nonlinear Schr\"{o}dinger equations},
 {\it Comm. Math. Phys.}, {\bf 85} (1982), 549-561.
\bibitem{ColinNA2004}
\textsc{M. Colin and L. Jeanjean },
\newblock{Solutions for a quasilinear Schr\"{o}dinger equation: a dual approach},
 {\it Nonlinear Anal.}, {\bf 56} (2004), 213-226.
\bibitem{ColinNon2010}
\textsc{M. Colin, L. Jeanjean and M. Squassina},
\newblock{Stability and instability results for standing waves of quasi-linear Schr\"{o}dinger equations},
 {\it Nonlinearity}, {\bf 23} (2010), 1353-1385.
\bibitem{Damascelli1998}
\textsc{L. Damascelli},
\newblock{Comparison theorems for some quasilinear degenerate elliptic operators and applications to symmetry and monotonicity results},
 {\it Ann. Inst. H. Poincar\'{e} C Anal. Non Lin\'{e}aire}, {\bf 15} (1998), 493-516.
\bibitem{DiazSIAM1994}
\textsc{J. I. Diaz and F. de Th\'{e}lin},
\newblock{On a nonlinear parabolic problem arising in some models related to turbulent flows},
 {\it SIAM J. Math. Anal.}, {\bf 25} (1994), 1085-1111.
\bibitem{Diaz1985}
\textsc{J. I. Diaz},
{\it Nonlinear partial differential equations and free boundaries, Vol. I},
 Elliptic equations, (1985).
\bibitem{JeanjeanNA1997}
\textsc{L. Jeanjean},
\newblock{Existence of solutions with prescribed norm for semilinear elliptic equations},
 {\it Nonlinear Anal.}, {\bf 28} (1997), 1633-1659.
\bibitem{JeanjeanCVPDE2020}
\textsc{L. Jeanjean and S.-S. Lu},
\newblock{A mass supercritical problem revisited},
 {\it Calc. Var. Partial Differential Equations}, {\bf 59} (2020), 1-43.
\bibitem{JeanjeanJMPA2017}
\textsc{L. Jeanjean and T. T. Le},
\newblock{Multiple normalized solutions for a Sobolev critical Schr\"{o}dinger equation},
 {\it Mathematische Annalen}, (2021), 1-34.
\bibitem{JeanjeanJMPA2022}
\textsc{L. Jeanjean, J. Jendrej, T. T. Le and N. Visciglia},
\newblock{Orbital stability of ground states for a Sobolev critical Schr\"{o}dinger equation},
 {\it J. Math. Pures Appl.},  {\bf 164} (2022), 158-179.
\bibitem{Ghoussoub1993}
\textsc{N. Ghoussoub},
{\it Duality and PerturbationMethods in Critical Point Theory},
Cambaridge University Press, (1993).
\bibitem{GueddaNA1989}
\textsc{M. Guedda and L. V\'{e}ron},
\newblock{Quasilinear elliptic equations involving critical Sobolev exponents},
 {\it Nonlinear Anal.}, {\bf 13} (1989), 879-902.
\bibitem{Kong2022}
\textsc{L. Kong and H. Chen},
\newblock{Normalized solutions for nonlinear Kirchhoff type equations in high dimensions},
 {\it Electron. Res. Arch.}, {\bf 30} (2022), 1282¨C1295.
\bibitem{LiJFPTA2021}
\textsc{H. Li and W. Zou},
\newblock{Normalized ground states for semilinear elliptic systems with critical and subcritical nonlinearities},
 {\it J. Fixed Point Theory Appl.}, {\bf 23} (2021), 1-30.
\bibitem{LionsAIHPAN1984}
\textsc{P. L. Lions},
\newblock{The concentration-compactness principle in the calculus of variations. The locally compact case. II},
 {\it Ann. Inst. H. Poincar\'{e} Anal. Non Lin\'{e}aire 1}, {\bf 1} (1984), 223-283.
\bibitem{SoaveJDE2020}
\textsc{N. Soave},
\newblock{Normalized ground states for the NLS equation with combined nonlinearities},
 {\it J. Differential Equations}, {\bf 269} (2020), 6941-6987.
\bibitem{SoaveJFA2020}
\textsc{N. Soave},
\newblock{Normalized ground states for the NLS equation with combined nonlinearities\: the Sobolev critical case},
 {\it J. Func. Anal.}, {\bf 279} (2020), 108610.
\bibitem{Willem1997}
\textsc{M. Willem},
{\it Minimax theorems},
 Springer Science \& Business Media, (1997).
\bibitem{YeMMAS2015}
\textsc{H. Ye.},
\newblock{The sharp existence of constrained minimizers for a class of nonlinear Kirchhoff equations},
 {\it Math. Methods Appl. Sci.}, {\bf 38} (2015), 2663-2679.
\bibitem{YeZMP2016}
\textsc{H. Ye},
\newblock{The mass concentration phenomenon for $L^2$-critical constrained problems related to Kirchhoff equations},
 {\it Z. Angew. Math. Phys.}, {\bf 67} (2016), 1-16.
\bibitem{YeZMP2021}
\textsc{H. Ye and Y. Yu},
\newblock{The existence of normalized solutions for $L^2$-critical quasilinear Schr\"{o}dinger equations},
 {\it J. Math. Anal. Appl.}, {\bf 497} (2021), 124839.
\end{thebibliography}

\end{document}